\documentclass{amsart}

\usepackage{amsthm, amssymb, amsmath, amscd}
\usepackage{eurosym}
\usepackage{amsfonts}
\usepackage{amsmath}
\usepackage{bbm}
\usepackage{setspace}
\usepackage{graphicx}
\usepackage{tensor}

\usepackage{color}

\usepackage[all,cmtip]{xy}

\usepackage{indentfirst}

\newtheorem{theorem}{Theorem}[section]
\newtheorem{thm*}{Theorem}

\newtheorem{cor}[theorem]{Corollary}

\newtheorem{lemma}[theorem]{Lemma}

\newtheorem{prop}[theorem]{Proposition}

\theoremstyle{definition}
\newtheorem{define}[theorem]{Definition}
\newtheorem{ex}[theorem]{Example}
\newtheorem{remark}[theorem]{Remark}
\newtheorem{notation}[theorem]{Notation}

\newcommand{\field}[1]{\mathbb{#1}}

\newcommand{\zkz}{\field{Z}/k\field{Z}}

\newcommand{\ch}{\mathrm{ch}}
\newcommand{\Z}{\field{Z}}
\newcommand{\C}{\field{C}}
\newcommand{\Q}{\field{Q}}

\DeclareFontFamily{OT1}{pzc}{}
\DeclareFontShape{OT1}{pzc}{m}{it}{<-> s * [1.1] pzcmi7t}{}
\DeclareMathAlphabet{\mathpzc}{OT1}{pzc}{m}{it}

\begin{document}
\title[Relative geometric assembly: Part II]{Relative geometric assembly and mapping cones\\
Part II: Chern characters and the Novikov property}
\author{Robin J. Deeley, Magnus Goffeng}
\date{\today}

\begin{abstract}
We study Chern characters and the assembly mapping for free actions using the framework of geometric $K$-homology. The focus is on the relative groups associated with a group homomorphism $\phi:\Gamma_1\to \Gamma_2$ along with applications to Novikov type properties. In particular, we prove a relative strong Novikov property for homomorphisms of hyperbolic groups and a relative strong $\ell^1$-Novikov property for polynomially bounded homomorphisms of groups with polynomially bounded cohomology in $\C$. As a corollary, relative higher signatures on a manifold with boundary $W$, with $\pi_1(\partial W)\to \pi_1(W)$ belonging to the class above, are homotopy invariant. 
\end{abstract}

\maketitle

\section{Introduction}

This paper is a continuation of ``Relative geometric assembly and mapping cones, Part I: The geometric model and applications" \cite{DGrelI}. The geometric constructions in \cite{DGrelI} were inspired by the analytic constructions in \cite{CWY}. In the present paper, we construct Chern characters on the relative $K$-homology groups and consider relative assembly at the level of homology. These constructions lead to (especially in specific examples) a better understanding of the relative assembly map associated with a group homomorphism $\phi:\Gamma_1\to \Gamma_2$. This assembly map relates the relative $K$-homology group $K_*(B\phi)$ with the relative $K$-theory group $K_{*+1}(C_\phi)$ of the mapping cone $C^*$-algebra $C_\phi$ constructed from $\phi$. The definitions of these groups and the assembly map are reviewed in Section \ref{sec:prel}. In \cite{DGrelI}, the relative assembly map was used to obtain vanishing results for manifolds with boundaries in the presence of positive scalar curvature metrics. In the present paper, we consider further applications such as the relative Novikov property of group homomorphisms $\phi:\Gamma_1\to \Gamma_2$.

As mentioned, an analytic version of relative assembly has been studied using localization algebras in \cite{CWY}. The approach in this paper is based on the geometric model for relative assembly in \cite{DGrelI}. This situation can be compared to models for Higson-Roe's analytic surgery group: the analytic model was originally constructed using coarse geometry \cite{HRsur1} and later using localization algebras \cite{xieyu}. A geometric model is constructed in \cite{DG} based on the ideas of relative constructions in geometric $K$-homology from \cite{DeeZkz, DeeRZ}. While the two analytically flavoured constructions of Higson-Roe's surgery group, via coarse geometry and localization algebras, come with the flexibility of $K$-theory, they often use non-separable $C^*$-algebras and defining geometric invariants thereon requires substantial work (see for instance  \cite{HReta,zenobisecondary}). The geometric model \cite{DG,DGII,DGIII} allows for explicit geometric invariants and, in particular, Chern characters defined at the level of cycles that are related to the higher index theorems of Wahl \cite{WahlAPSForCstarBun} and Lott \cite{Lott}. The results in the present paper for the assembly of relative $K$-groups are closely modelled on the results in \cite{DGIII} for Higson-Roe's surgery group.

Our main results centre on the relationship between relative $K$-groups and relative homology through the construction of Chern characters. Chern characters have previously been studied in this context in \cite{BD, DGIII} and its relation to Baum-Connes' assembly mapping in \cite{bcchern}. A relative Chern character of Dirac operators on manifolds with boundaries was also studied in \cite{leschmopfl} using relative JLO-cocycles in cyclic theory. The underlying philosophy fits into Baum's approach to index theory, see \cite[page $154$]{BD} and the discussion in \cite{BE}. We recall its salient points below in Subsection \ref{baumsapp}. Before doing so, we state the main results.

\subsection{Main results}
The basic setup of the paper is as follows. Let $\Gamma_1$ and $\Gamma_2$ be countable discrete groups and $\phi:\Gamma_1\to \Gamma_2$ a group homomorphism. We fix Fr\'echet algebra completions $\mathcal{A}_i$ of $\C[\Gamma_i]$, $i=1,2$, such that $\phi$ extends to a continuous $*$-homomorphism $\phi_\mathcal{A}:\mathcal{A}_1\to \mathcal{A}_2$. In applications, we often assume that $\mathcal{A}_i$ are continuously embedded as dense $*$-subalgebras of some $C^*$-algebras, i.e. admit faithful $*$-representations on a Hilbert space. It simplifies the situation, but is not needed for the general theory, to assume that $\phi_\mathcal{A}$ is continuous in the topology of the ambient $C^*$-algebras. The examples of ambient $C^*$-algebras to keep in mind are the maximal completions $C^*_{\bf max}(\Gamma_1)\to C^*_{\bf max}(\Gamma_2)$ and the reduced completions $C^*_{\bf red}(\Gamma_1)\to C^*_{\bf red}(\Gamma_2)$. 

We make use of the absolute Chern character $\ch^\mathcal{A}:K_*^{\rm geo}(pt;\mathcal{A})\to \hat{H}^{\rm rel}_*(\mathcal{A})$ previously defined in \cite[Section 4]{DGIII} (we recall the construction below in Section \ref{absolutsubse}).  Here $K_*^{\rm geo}(pt;\mathcal{A})$ and $\hat{H}^{\rm rel}_*(\mathcal{A})$ denote the geometric $K$-homology and topological de Rham homology, respectively, of $\mathcal{A}$ as defined in \cite{DG} and \cite[Section 4]{DGIII}, respectively. There are corresponding relative groups $K_*^{\rm geo}(pt;\phi_\mathcal{A})$ and $\hat{H}^{\rm rel}_*(\phi_\mathcal{A})$ that we introduce in Subsubsection \ref{remarkongengeomodel} and Subsection \ref{relativedersubs}. 

\begin{thm*}[See Theorem \ref{cheronkthe} on page \pageref{cheronkthe}]
There is a Chern character $\ch^{\phi_\mathcal{A}}_{\rm rel}:K_*^{\rm geo}(pt;\phi_\mathcal{A})\to \hat{H}^{\rm rel}_*(\phi_\mathcal{A})$ that fits into a commuting diagram with exact rows:
\begin{center}
\scriptsize
$\begin{CD}
@>>> K_*^{\rm geo}(pt; \mathcal{A}_1) @>(\phi_\mathcal{A})_*>> K_*^{\rm geo}(pt;\mathcal{A}_2) @>r>> K_*^{\rm geo}(pt;\phi_\mathcal{A}) @>\delta>> K_{*-1}^{\rm geo}(pt;\mathcal{A}_1) @>>> \\
@. @VV\ch^{\mathcal{A}_1}V @VV\ch^{\mathcal{A}_2}V  @VV\ch_{\rm rel}^{\phi_\mathcal{A}} V  @VV\ch^{\mathcal{A}_1} V\\
@>>> \hat{H}_*^{\rm dR}(\mathcal{A}_1) @>(\phi_\mathcal{A})_*>>\hat{H}_*^{\rm dR}(\mathcal{A}_2) @>>> \hat{H}_*^{\rm rel}(\phi_\mathcal{A}) @>>> \hat{H}_*^{\rm dR}(\mathcal{A}_1)@>>> \\
\end{CD}$
\end{center}
\end{thm*}

Fr\'echet algebra completions of the group algebra allow for definitions of both absolute and relative assembly maps for free actions. These mappings take values in suitable geometric $K$-homology groups of a point; precise definitions are given in Section \ref{subseonrelassforfre}, following \cite{DGrelI}. These free assembly maps are denoted by $\mu^\mathcal{A}_{\textnormal{geo}} :K_*(B\Gamma)\to K_*^{\rm geo}(pt; \mathcal{A})$ in the absolute case and $\mu^{\phi_\mathcal{A}}_{\textnormal{geo}} :K_*(B\phi)\to K_*^{\rm geo}(pt; \phi_\mathcal{A})$ in the relative case, compare to the $C^*$-algebraic setup in \cite[Section 3]{DGrelI}. In Section \ref{assinhom} we define assembly mappings at the level of homology denoted by $\mu^\mathcal{A}_{\rm M}:H_*(B\Gamma)\to \hat{H}^{\rm dR}_*(\mathcal{A})$ in the absolute case and $\mu^{\phi_\mathcal{A}}_{\rm M}:H_*(B\phi)\to \hat{H}^{\rm rel}_*(\phi_\mathcal{A})$ in the relative case.

\begin{thm*}[See Theorem \ref{thmonassembly} on page \pageref{thmonassembly}]
The following diagram commutes
\begin{center}
$\begin{CD}
K^{\textnormal{geo}}_*(B \phi) @>\mu_{\textnormal{geo}}^{\phi_{\mathcal{A}}} >> K^{\textnormal{geo}}_*(pt;\phi_{\mathcal{A}})  \\
@V\mathrm{ch}^{B\phi}_*V
V @VV\mathrm{ch}^{\phi_{\mathcal{A}}}_*V   \\
H_*(B\phi)@>\mu_{\rm M}^{\phi_\mathcal{A}} >> \hat{H}^{\rm rel}_*(\phi_{\mathcal{A}}) \\
\end{CD}$
\end{center}
Here $\mathrm{ch}^{B\phi}_*:K^{\textnormal{geo}}_*(B \phi) \to H_*(B\phi)$ denotes the relative Chern character (see Subsection \ref{subseconrelhomonkjom} below). The analogous diagram for the absolute assembly mappings also commutes.
\end{thm*}

A natural question is what conditions imply that the relative free assembly map is injective. Whenever $\mu^\phi_{\rm geo}:K_*^{\rm geo}(B\phi) \to K_*^{\rm geo}(pt;\phi:C^*_\epsilon(\Gamma_1)\to C^*_\epsilon(\Gamma_2))$ is injective (for some $C^*$-completion $\epsilon$) we say that $\phi$ has the $\epsilon$-strong relative Novikov property. As was shown in \cite[Theorem 4.18]{DGrelI}, the strong relative Novikov property implies the relative Novikov property as discussed in \cite[Section 12]{weinnov} and \cite[Section 4.9]{Lott}. The relative Novikov property for $\phi$ concerns homotopy invariance of relative higher signatures on manifolds with boundary where the boundary inclusion factors over $\phi$ at the level of fundamental groups, for details see \cite[Section 4.3]{DGrelI} and Remark \ref{strongremark} (see page \pageref{strongremark}). When $\Gamma_1$ and $\Gamma_2$ are torsion-free, the five lemma guarantees that $\phi$ has the $\epsilon$-strong Novikov property if $\Gamma_1$ and $\Gamma_2$ satisfies the $\epsilon$-Baum-Connes conjecture, in this case $\mu_\phi^{\rm geo}$ will even be an isomorphism. In the presence of torsion, there is currently no known relationship between the strong relative Novikov property of $\phi$ and the validity of the Baum-Connes conjecture for $\Gamma_1$ and $\Gamma_2$ respectively, a question we aim to address in future work. 

If the group algebras $\C[\Gamma_1]$ and $\C[\Gamma_2]$ admit suitable completions $\mathcal{A}_1$ and $\mathcal{A}_2$ it is possible to prove the strong relative Novikov property using Chern characters along the same lines as in \cite{cmnovikov}. A precise condition is provided in Theorem \ref{theoremonabscond} and Corollary \ref{cornoviko}. We apply this to hyperbolic groups in Theorem \ref{rnovikovforhyp} and to groups with polynomially bounded cohomology in $\C$ in Theorem \ref{polbddthm}. These results read as follows.

\begin{thm*}
\label{apptonovi}
A homomorphism of hyperbolic groups that is continuous in the reduced $C^*$-norm has the reduced strong relative Novikov property. Moreover, any polynomially bounded homomorphism of groups with polynomially bounded cohomology in $\C$ (see Definition \ref{pcandpolbdd}) will have the $\ell^1$-strong relative Novikov property.
\end{thm*}

The reader should note that a group homomorphism is continuous in the reduced $C^*$-norm if and only if its kernel is amenable. For hyperbolic groups, amenable subgroups are virtually cyclic by \cite[Theorem 38, page 21]{ghysdelaharpe}. In other words, the first stated condition in Theorem \ref{apptonovi} is equivalent to the group homomorphism having virtually cyclic kernel.

The Novikov property using Banach algebras has recently attracted some attention. For instance in \cite{engelnovi}, the Novikov property for $\ell^1(\Gamma)$ was proven for $\Gamma$ being $F_\infty$ and polynomially contractible -- a condition slightly stronger than having polynomially bounded cohomology in $\C$ (as in Definition \ref{pcandpolbdd}). The result stated in Theorem \ref{apptonovi} will in the absolute case imply the $\ell^1$-strong Novikov property for groups with polynomially bounded cohomology in $\C$. {\it The reader should be wary} of the fact that we throughout the paper only consider the geometric version of assembly.

\subsection{Context in Baum's approach to index theory}
\label{baumsapp}
We briefly review Baum's approach to index theory as discussed in \cite{BD}. The starting point for this approach is two realizations of the generalized homology theory $K$-homology. For simplicity, suppose $X$ is a compact manifold. The geometric $K$-homology of $X$, denoted by $K^{\rm geo}_*(X)$, is obtained from cycles of the form $(M, E, f)$ modulo a geometrically defined relation, see \cite{BD} for details; here $M$ is a smooth closed ${\rm spin^c}$-manifold, $E$ is a smooth Hermitian vector bundle over $M$, and $f:M \rightarrow X$ is a continuous map. The analytic $K$-homology of $X$, denoted by $K^{\rm ana}_*(X)$, is the Kasparov group $KK^*(C(X), \C)$ \cite{Kas}. Localization algebras can also be used to obtain an analytic realization of $K$-homology, see \cite{qr10,Yu}. The cycles in Kasparov's realization of $K$-homology are Fredholm modules see \cite[Chapter 8]{HR} for details. There is an isomorphism:
\[ 
\lambda : K_*^{\rm geo}(X) \rightarrow K_*^{\rm ana}(X)
\]
defined explicitly at the level of geometric cycles via $(M, E, f) \mapsto f_*([D_E])$ where $[D_E]\in K_*^{\rm ana}(M)$ denotes the analytic $K$-homology class associated to the Dirac operator twisted by $E$.

Although, $\lambda$ is an isomorphism at the level of classes, its direction as a map on {\it cycles} has important consequences in index theory. A particularly relevant one for the present paper is the definition of Chern characters. While the Chern character can be defined at the level of geometric cycles (upon changing to a model where the cycles contain connection data), there is no definition of the Chern character as a map defined at the level of Fredholm modules. Nevertheless, one has (see \cite[Part 5]{BD} for details) the following diagram:
\begin{equation}
\label{cherndefka}
 \begin{CD} K^{\rm geo}_*(X) @>\lambda >> K^{\rm ana}_*(X)  \\
@V{\rm ch}VV @. \\ 
H_*^{\rm dR}(X)
\end{CD}
\end{equation}
where $H_*^{\rm dR}(X)$ denotes the de Rham homology of $X$. Therefore, using the fact that $\lambda$ is an isomorphism, one can define the Chern character of a class in the analytic $K$-homology group to be the Chern character of its preimage under $\lambda$. Finding an explicit preimage of particular analytic $K$-homology classes is the goal of Baum's approach to index theory. Its solution implies an index formula via the Chern character.  

The Chern characters constructed in the present paper are motivated by the above and as mentioned are similar to constructions in \cite{DGIII}.

\subsection{Notation}
We now list the notation used in the paper. The list is complemented by the list of notation in part I of this series of papers \cite{DGrelI}. As some of it has been introduced above, the reader can skip this section and return to it only when needed.

We are interested in the relative homology of a mapping $h:Y\to X$. It will sometimes be denoted by $[h:Y\to X]$. If $h_i:Y_i\to X_i$, $i=1,2$,  and $g:Y_1\to Y_2$ and $f:X_1\to X_2$ are continuous mappings such that $h_2\circ g=f\circ h_1$ we write 
$$(f,g):[h:Y_1\to X_1]\to [h_2:Y_2\to X_2].$$
We consider group homomorphisms $\phi:\Gamma_1\to \Gamma_2$ of discrete groups. Their classifying spaces are denoted by $B\Gamma_i$ and the associated $\Gamma_i$-space by $E\Gamma_i$. We let $B\phi:B\Gamma_1\to B\Gamma_2$ denote the continuous mapping associated with $\phi$.

We will denote Fr\'echet *-algebras by $\mathcal{A}$. We are interested in the relative setting for a continuous *-homomorphism $\phi_\mathcal{A}:\mathcal{A}_1\to \mathcal{A}_2$. The subscript on $\phi_\mathcal{A}$ is to distinguish it from a group homomorphism $\phi:\Gamma_1\to \Gamma_2$. We sometimes drop the subscript on $\phi_\mathcal{A}$ when the distinction is clear from context. In the cases of interest to us, $\mathcal{A}$ is continuously embedded as a dense $*$-subalgebra of a $C^*$-algebra. Often, but not exclusively, we assume that $\mathcal{A}$ is closed under holomorphic functional calculus in the ambient $C^*$-algebra, whenever we do so it will be explicitly stated. $C^*$-algebras are denoted by $A$ and $B$ and for a $*$-homomorphism $\phi:A\to B$, $C_\phi$ denote its mappings cone, which is the $C^*$-algebra
$$C_\phi:=\{a\oplus b\in A\oplus C_0((0,1],B): \; b(1)=\phi(a)\}.$$ 

A homomorphism of abelian groups is said to be a rational/complex isomorphism if it becomes an isomorphism upon tensoring with $\Q$ and $\C$, respectively.

\section{Preliminaries}
\label{sec:prel}

We review the construction of geometric $K$-homology and noncommutative de Rham homology in this section. We also recall the construction of Chern characters in the absolute setting from \cite{DGIII}. The length of this section is justified by the fact that most results are spread out in the literature or are found only in the $C^*$-algebraic setting.

The results in geometric $K$-homology are well known for coefficients in a $C^*$-algebra, see \cite{BD, Rav, Wal}, also in the relative setting \cite{DeeMappingCone, DGrelI} and to some extent for Banach and Fr\'echet algebras \cite[Appendix]{DG}. We mainly use Fr\'echet algebras where one expects more invariants on its de Rham homology and a relation to $C^*$-algebra $K$-theory via a dense inclusion. Indeed, the construction of a well behaved Chern character on a $C^*$-algebra depends on the choice of a dense subalgebra analogously to the classical constructions in Chern-Weil theory. The need for choosing dense subalgebras is based on work of Johnson \cite{bejohn}. Johnson showed that the standard approaches to homology lack interesting invariants when restricting to $C^*$-algebras. There are more sophisticated tools for dealing with homology of $C^*$-algebras, e.g. Meyer's analytic cyclic homology \cite{meyer} and Puschnigg's local cyclic homology \cite{puschnigg}, which require further technical considerations. For our purposes, the simpler setting of homology for the choice of a dense $*$-subalgebra suffices.

Fr\'echet *-algebras will be denoted by $\mathcal{A}$. The examples we have in mind come from completions of group algebras. The reader is encouraged to keep the following examples in mind. 

\begin{ex}
\label{firstappofrd}
Let $\Gamma$ be a finitely generated discrete group with a length function $L:\Gamma\to \field{R}_{\geq 0}$. The associated Sobolev spaces $H^s_L(\Gamma)$, $s\in \field{R}$ are defined as
$$H^s_L(\Gamma):=\left\{f:\Gamma\to \C: \; \|f\|_{H^s_L(\Gamma)}:=\sqrt{\sum_{\gamma\in \Gamma} (1+L(\gamma))^{2s}|f(\gamma)|^2}<\infty\right\}.$$
The length function is said to be of rapid decay if for some $s$, the identity mapping on the group algebra extends to a continuous inclusion $H^s_L(\Gamma)\to C^*_{\bf red}(\Gamma)$. If $\Gamma$ admits a length function of rapid decay, we say $\Gamma$ has property (RD) and we fix such a length function on $\Gamma$. See more in \cite{Chatterjeeeee,Jollisiantttt}. Hyperbolic groups have property (RD) in word length functions associated with symmetric generating sets, see \cite{harpesart}. An amenable group has property (RD) if and only if it is of polynomial growth, see \cite[Proposition B]{Jollisiantttt}. The Fr\'echet space of interest is the Schwartz space $H^\infty_L(\Gamma):=\cap_{s>0}H^s_L(\Gamma)$. The Schwartz space is a Fr\'echet $*$-algebra continuously embedded as a dense $*$-subalgebra of $C^*_{\bf red}(\Gamma)$ closed under holomorphic functional calculus if $\Gamma$ has property (RD), see \cite[Lemma 1.2.4]{Jollisiantttt}. 

Returning to the relative situation, let $\Gamma_1$ and $\Gamma_2$ be groups with length functions of rapid decay $L_1$ and $L_2$, respectively. A homomorphism $\phi:\Gamma_1\to \Gamma_2$ is said to have polynomial growth if there are $C,p>0$ such that $L_2(\phi(\gamma))\leq CL_1(\gamma)^p$ for all $\gamma\in \Gamma_1$. If $\phi$ has polynomial growth the induced $*$-homomorphism $\phi_{H^\infty_L}:H^\infty_{L_1}(\Gamma_1)\to H^\infty_{L_2}(\Gamma_2)$ is well defined and continuous. 
\end{ex}

\begin{ex}
\label{ageoneess}
For any group $\Gamma$ we can form the Banach $*$-algebra $\ell^1(\Gamma)$. The $\ell^1$-algebra depends functorially on $\Gamma$. Any homomorphism $\phi:\Gamma_1\to \Gamma_2$ induces a continuous $*$-homomorphism $\phi_{\ell^1}:\ell^1(\Gamma_1)\to \ell^1(\Gamma_2)$. The Bost conjecture states that the $\ell^1$-assembly $K_*^\Gamma(\underline{\mathcal{E}}\Gamma)\to K_*(\ell^1(\Gamma))$ is an isomorphism; the Bost conjecture holds as soon as the $\gamma$-element equals $1$ in $KK^{\rm Ban}$. There are to date no known counterexamples to this conjecture and it is known in all the known cases where the Baum-Connes conjecture holds. The reader can see \cite{laffaICM} for further details. 

It is possible to refine $\ell^1(\Gamma)$ after choosing a length function $L$ on $\Gamma$. We define the Banach algebras
$$H^{1,s}_L(\Gamma):=\left\{f:\Gamma\to \C: \; \|f\|_{H^{1,s}_L(\Gamma)}:=\sum_{\gamma\in \Gamma} (1+L(\gamma))^{s}|f(\gamma)|<\infty\right\}.$$
We note that $\Gamma$ has polynomial growth if and only if there are $s,s_0\in \mathbb{R}$ such that $H^{s+s_0}(\Gamma)\subseteq H^{1,s}(\Gamma)$. The converse inclusion $H^{s}(\Gamma)\supseteq H^{1,s}(\Gamma)$ holds for any $s$. Following \cite[Section 2]{jolissatwo} we define the Fr\'echet algebra $H^{1,\infty}_L(\Gamma):=\cap_{s\geq 0} H^{1,s}_L(\Gamma)$. By \cite[Proposition 2.3]{jolissatwo}, $H^{1,\infty}_L(\Gamma)$ is closed under holomorphic functional calculus in $\ell^1(\Gamma)$ so the inclusion $H^{1,\infty}_L(\Gamma)\hookrightarrow \ell^1(\Gamma)$ induces an isomorphism on $K$-theory. In particular, if $\Gamma$ satisfies the Bost conjecture and the Baum-Connes conjecture, the mapping $K_*(H^{1,\infty}_L(\Gamma))\to K_*(C^*_{\bf red}(\Gamma))$ induced by the inclusion is an isomorphism. Finally, if $\phi:\Gamma_1\to \Gamma_2$ is a group homomorphism of at most polynomial growth the induced $*$-homomorphism $\phi_{H^{1,\infty}_L}:H^{1,\infty}_{L_1}(\Gamma_1)\to H^{1,\infty}_{L_2}(\Gamma_2)$ is well defined and continuous. 
\end{ex}

\begin{ex}
A standard construction of dense $*$-subalgebras of a $C^*$-algebra that are closed under holomorphic functional calculus comes from domains of closed derivations. See more in for instance \cite[Section 1.2]{bradavijen}. Let $A$ be a $C^*$-algebra. Assume that $\mathcal{I}$ is a Banach $*$-algebra with a contractive bimodule action of $A$ such that $(aj)^*=j^*a^*$ for $a\in A$, $j\in \mathcal{I}$. A closed densely defined $*$-derivation $\delta:A\to \mathcal{I}$ is a densely defined linear mapping satisfying the Leibniz rule $\delta(ab)=a\delta(b)+\delta(a)b$ and $\delta(a^*)=-\delta(a)^*$. The space $\mathcal{A}_\delta:=\mathrm{Dom}(\delta)$ is a Banach $*$-algebra in the $*$-algebra structure inherited from $A$ and the norm $\|a\|_{\mathcal{A}_\delta}:=\|a\|_A+\|\delta(a)\|_{\mathcal{I}}$. In the case that $\mathcal{I}$ is matrix normed\footnote{That is, when we have specified a choice of norms $(\|\cdot\|_{n,\mathcal{I}})$ on all matrices $M_n(\mathcal{I})$ making the bimodule action of the unitalization $M_n(\tilde{\mathcal{I}})$ on $M_n(\mathcal{I})$ contractive.} and the $A$-action is completely bounded\footnote{That is, when the norm of the bimodule action of $M_n(A)$ on $M_n(\mathcal{I})$ is uniformly bounded in $n$.}, $\mathcal{A}_\delta$ can also be matricially normed by means of the pull back norm along the continuous homomorphism
$$\pi_\delta:a\mapsto \begin{pmatrix} a& 0\\ \delta(a)& a\end{pmatrix}, \quad \pi_\delta:\mathcal{A}_\delta\to \left\{\begin{pmatrix} T_{11}& T_{12}\\ T_{21}& T_{22}\end{pmatrix}: T_{11}, T_{22}\in CB(\mathcal{I}), \; T_{12},T_{21}\in \mathcal{I}\right\} .$$
Here $CB(\mathcal{I})$ denotes the completely bounded multipliers of $\mathcal{I}$ consisting of the set of multipliers $m$ of $\mathcal{I}$ such that the norm of $m\otimes 1_{M_n(\C)}$ acting on $M_n(\mathcal{I})$ is uniformly bounded in $n$. It follows from \cite[Proposition 3.12]{blackcuntz} that $\mathcal{A}_\delta\subseteq A$ is closed under holomorphic functional calculus. Similar ideas were used in \cite{pushing} to construct holomorphically closed subalgebras of $C^*_{\bf red}(\Gamma)$, for a hyperbolic group $\Gamma$, on which a trace is continuous. While dense $*$-subalgebras defined as domains of $*$-derivations are theoretically tractable, their functoriality properties and cocycles are difficult to study.

Such derivations arise from bounded and unbounded Fredholm modules. We denote a bounded Fredholm module on $A$ by $(\pi,\mathcal{H},F)$ and an unbounded by $(\pi,\mathcal{H},D)$, the reader can find the definitions thereof in \cite[Introduction]{GM}, for instance. For an unbounded Fredholm module we set $\mathcal{I}=\mathbb{B}(\mathcal{H})$. For any $*$-algebra $\mathcal{A}\subseteq \mathrm{Lip}_D(A):=\{a\in A:\, a\mathrm{Dom}(D)\subseteq \mathrm{Dom}(D), \;[D,a] \mbox{  bounded}\}$, which is dense in $A$, we obtain a closed $*$-derivation $\delta_{D,\mathcal{A}}:A\dashrightarrow \mathbb{B}(\mathcal{H})$ from closing the $*$-derivation $\mathcal{A}\ni a\mapsto \overline{[D,a]}\in \mathbb{B}(\mathcal{H})$. For a bounded Fredholm module we can take $\mathcal{I}$ to be a symmetrically normed ideal of compact operators (see \cite[Chapter 1.7]{simontrace}) such that $F^*-F,F^2-1\in \mathcal{I}^2$. For any $*$-algebra $\mathcal{A}\subseteq \mathrm{Sum}_\mathcal{I}(A):=\{a\in A:\, \;[F,a]\in \mathcal{I}\}$, we define $\delta_{F,\mathcal{A}}$ as the closure of $\mathcal{A}\ni a\mapsto [F,a]\in \mathcal{I}$ in $A$. This is a densely defined derivation on the $C^*$-algebra closure $\overline{\mathcal{A}}$. By definition, a Fredholm module is $\mathcal{I}$-summable if $F^*-F,F^2-1\in \mathcal{I}^2$ and $a\mapsto [F,a]\in \mathcal{I}$ is densely defined on $A$. For $\mathcal{I}$-summable Fredholm modules one can assume that $\mathcal{A}\subseteq A$ is dense and arrive in the situation of the paragraph above. 

A concrete example of a spectral triple arising on discrete groups comes from length functions. Let $\Gamma$ be a discrete group with a proper length function $L$. We define $D_L$ as the positive self-adjoint multiplication operator on $\ell^2(\Gamma)$ defined from $L$. By definition, $\mathrm{Dom}(D_L^s)=H^s_L(\Gamma)$ for any $s\geq 0$. The associated Connes-Moscovici algebra $\mathcal{B}^\infty_L(\Gamma)$ is associated with the derivation $a\mapsto [D_L,a]$. The operation $\delta_L:=[D_L,\cdot]:\mathbb{B}(\ell^2(\Gamma))\dashrightarrow \mathbb{B}(\ell^2(\Gamma))$ is partially defined with domain $\mathrm{Dom}(\delta_L):=\mathrm{Lip}_D(\mathbb{B}(\ell^2(\Gamma)))$. For $k\geq 1$, we define the Banach algebra $\mathcal{B}^k_L(\Gamma):=\mathrm{Dom}(\delta_L^k)\cap C^*_{\bf red}(\Gamma)$ and the Fr\'echet algebra $\mathcal{B}^\infty_L(\Gamma):=\cap_{k\geq 1} \mathcal{B}^k_L(\Gamma)$. It is easily checked that $\C[\Gamma]\subseteq \mathcal{B}^\infty_L(\Gamma)$, so $\mathcal{B}^\infty_L(\Gamma)\subseteq C^*_{\bf red}(\Gamma)$ is dense. By \cite[Proposition 3.12]{blackcuntz}, $\mathcal{B}^k_L(\Gamma)\subseteq C^*_{\bf red}(\Gamma)$ is closed under holomorphic functional calculus for any $k\in \mathbb{N}_+\cup\{\infty\}$. In fact, if we let $\delta_e\in \ell^2(\Gamma)$ denote the delta function in $e\in\Gamma$, we have
$$\|\delta_L^k(a)\delta_e\|_{\ell^2(\Gamma)}=\|a\|_{H^{k}_L(\Gamma)}.$$
See \cite[Proof of Lemma 6.4]{cmnovikov}. In particular, there is a continuous inclusion $\mathcal{B}^k_L(\Gamma)\subseteq H^{k}_L(\Gamma)$. The Connes-Moscovici algebra is to some extent functorial: let $\phi:\Gamma_1\to \Gamma_2$ be a group inclusion and $L$ a length function on $\Gamma_2$. We pick a transversal $S\subseteq \Gamma_2$ for $\Gamma_2/\Gamma_1$. The inclusion $\phi$ induces a continuous $*$-monomorphism $\phi:\cap_{\gamma \in S}\mathcal{B}^\infty_{\gamma^*L}(\Gamma_1)\to \mathcal{B}^\infty_L(\Gamma_2)$, where $\mathcal{B}^\infty_{\gamma^*L}(\Gamma_1)$ is defined as above as the infinite domain of the derivation that the multiplication operator $\gamma^*L=L(\gamma\,\cdot):\Gamma_1\to \mathbb{R}_+$ defines.

\end{ex}

\subsection{Geometric $K$-homology and Fr\'echet algebra coefficients}

We will make use of geometric models for $K$-homology to place $K$-theory and $K$-homology on an equal, geometric footing. For a $C^*$-algebra $A$, the geometric $K$-homology of a space $X$ with coefficients in $A$, $K_*^{\rm geo}(X;A)$ has been well studied, see for instance \cite{Rav, Wal}. It is a well known fact that the map (see Theorem \ref{closedandkhom}) $\lambda:K_*^{\rm geo}(X;A)\to KK_*(C(X),A)$ is an isomorphism for a finite $CW$-complex $X$, the proof goes as in \cite{BHS}. Whenever $X$ is a locally finite $CW$-complex, the same mapping is an isomorphism to $\varinjlim_{Y\subseteq X\mbox{\scriptsize compact}} KK_*(C(Y),A)$ -- the compactly supported version of $KK_*(C_0(X),A)$. 

The construction of $K_*^{\rm geo}(X;\mathcal{A})$ for a Fr\'echet $*$-algebra $\mathcal{A}$ is carried out similarly and many properties continue to hold for Fr\'echet $*$-algebras. The main difference is that there is in general no obvious relation to the $K$-theory of $\mathcal{A}$ unless $\mathcal{A}$ is a Banach $*$-algebra (see \cite[Appendix]{DG}) or more generally if $\mathcal{A}$ is closed under holomorphic functional calculus in a Banach $*$-algebra (see Theorem \ref{closedandkhom}). 

A locally trivializable bundle $\mathcal{E}_{\mathcal{A}}\to X$ of finitely generated projective $\mathcal{A}$-modules will be referred to as an $\mathcal{A}$-bundle. If $X$ is a manifold, we can up to isomorphism assume that $\mathcal{E}_\mathcal{A}$ is smooth, see \cite[Theorem 3.14]{SchL2}. The prototypical example of such a bundle is the Mischchenko bundle, see Equation \eqref{michscdef} on page \pageref{michscdef}.

\begin{define}
Let $X$ be a topological space. A geometric cycle on $X$ is a triple $(M,\mathcal{E}_\mathcal{A},f)$ where
\begin{enumerate}
\item $M$ is a closed spin$^c$-manifold;
\item $\mathcal{E}_\mathcal{A}\to M$ is an $\mathcal{A}$-bundle;
\item $f:M\to X$ is a continuous mapping.
\end{enumerate}
If $(W,\mathcal{E}_\mathcal{A},f)$ is as above but $W$ is allowed to have boundary, we say that $(W,\mathcal{E}_\mathcal{A},f)$ is a geometric cycle with boundary or a bordism in geometric $K$-homology. We write $\partial (W,\mathcal{E}_\mathcal{A},f):=(\partial W,\mathcal{E}_\mathcal{A}|_{\partial W},f|_{\partial W})$.
\end{define}

The set of isomorphism classes of geometric cycles on $X$ with coefficients in $\mathcal{A}$ forms an abelian semigroup under disjoint union of cycles. This set comes equipped with a Baum-Douglas relation: the equivalence relation generated by the equivalence relations disjoint union/direct sum, bordism and vector bundle modification. See more in \cite{BD}. We define $K_*^{\rm geo}(X;\mathcal{A})$ as the set of equivalence classes. It is a $\mathbb{Z}/2$-graded set, graded by the dimension of the cycle modulo $2$. The disjoint union operation makes $K_*^{\rm geo}(X;\mathcal{A})$ into a $\mathbb{Z}/2$-graded abelian group. The group $K_*^{\rm geo}(X;\mathcal{A})$ is in general difficult to compute, here we give some cases where it is computable.

\begin{theorem}
\label{closedandkhom}
Let $\mathcal{A}$ be a Fr\'echet $*$-algebra.
\begin{enumerate}
\item If $\mathcal{A}=A$ is a $C^*$-algebra, then for any finite $CW$-complex $X$ the following mapping is an isomorphism:
\begin{align*}
\lambda:K_*^{\rm geo}(X;A)&\to KK_*(C(X),A), \\
& (M,\mathcal{E}_\mathcal{A},f)\mapsto [f]\otimes_{C(M)}[[\mathcal{E}_\mathcal{A}]]\otimes_{C(M)} [D_M],
\end{align*}
where $[f]\in KK_0(C(X),C(M))$ denotes the $*$-homomorphism associated with $f$, $[[\mathcal{E}_\mathcal{A}]]\in KK_0(C(M),C(M)\otimes \mathcal{A})$ the bimodule associated with $\mathcal{E}_\mathcal{A}$ and $[D_M]\in KK_{\mathrm{dim}(M)}(C(M),\mathbb{C})$ the spin$^c$-Dirac operator on $M$. More generally, for a \emph{locally finite $CW$-complex} $X$, 
$$\lambda:K_*^{\rm geo}(X;A)\xrightarrow{\sim} KK_*^{\rm c}(C_0(X),A):=\varinjlim KK(C(Y),A),$$
where the direct limit is taken over all compact subsets $Y\subseteq X$. 
\item If $\mathcal{A}$ is a Banach $*$-algebra, the natural mapping $K_*(\mathcal{A})\to K_*^{\rm geo}(pt;\mathcal{A})$ that maps a finitely generated projective module $\mathcal{E}_\mathcal{A}$ to the cycle $(pt,\mathcal{E}_\mathcal{A})$ is an isomorphism.
\item If $\mathcal{A}\subseteq \tilde{\mathcal{A}}$ is a dense inclusion of Fr\'echet $*$-algebras closed under holomorphic functional calculus, the functorially associated mapping $K_*^{\rm geo}(X,\mathcal{A})\to K_*^{\rm geo}(X,\tilde{\mathcal{A}})$ is an isomorphism. In particular, if $\tilde{\mathcal{A}}=A$ is a $C^*$-algebra and $X$ a finite $CW$-complex, $K_*^{\rm geo}(X;\mathcal{A})\cong KK_*(C(X),A)$.
\end{enumerate}
\end{theorem}

\begin{proof}
Part 1 is proven in \cite[Section 2.3]{Wal}. Part 2 is proven in \cite[Appendix]{DG}. To prove 3, we note that under the given assumptions, for any compact space $Y$, the inclusion $C(Y,\mathcal{A})\subseteq C(Y,\tilde{\mathcal{A}})$ is dense and closed under holomorphic functional calculus. In particular, the inclusion $C(Y,\mathcal{A})\subseteq C(Y,\tilde{\mathcal{A}})$ induces an isomorphism between the monoids of isomorphism classes of finitely generated projective modules over $C(Y,\mathcal{A})$ and $C(Y,\tilde{\mathcal{A}})$, respectively. As such, any $\tilde{\mathcal{A}}$-bundle on a compact space $Y$ is isomorphic to one of the form $\mathcal{E}_{\mathcal{A}}\otimes_\mathcal{A}\tilde{\mathcal{A}}$ and $\mathcal{E}_\mathcal{A}$ is determined uniquely up to isomorphism of $\mathcal{A}$-bundles. Therefore, the sets of isomorphism classes of geometric cycles on $X$ are the same for the two coefficients $\mathcal{A}$ and $\tilde{\mathcal{A}}$, respectively.
\end{proof}

\begin{remark}
If there is a mapping $\phi:\mathcal{A}_1\to \mathcal{A}_2$ that induces an isomorphism on $K$-theory, and $\mathcal{A}_2$ is holomorphically closed in a Banach $*$-algebra, Theorem \ref{closedandkhom}.3 shows that the mapping $K_*(\mathcal{A}_1)\to K_*^{\rm geo}(pt;\mathcal{A}_1)$ from Theorem \ref{closedandkhom}.2 is an injection.
\end{remark}

\begin{remark}
\label{indasdef}
For a cycle $(M,\mathcal{E}_A)$ defining a class in $K_*(pt; A)$ where $A$ is a $C^*$-algebra, we let $\mathrm{ind}_{\rm AS}(M,\mathcal{E}_A)\in K_*(A)$ denote the image of $\lambda(M,\mathcal{E}_A)\in KK_*(\mathbb{C},A)$ under $KK_*(\mathbb{C},A)\cong K_*(A)$. More concretely, $\mathrm{ind}_{\rm AS}(M,\mathcal{E}_A)\in K_*(A)$ is the $A$-valued index of an $A$-linear Dirac operator $D_{\mathcal{E}_A}^M$ acting on sections of $S_M\otimes \mathcal{E}_A\to M$, where $S_M\to M$ denotes the complex spinor bundle that the spin$^c$-structure defines.
\end{remark}

\begin{theorem}
\label{sameasinappbutfre}
Let $\mathcal{A}$ be a Fr\'echet $*$-algebra. The functor $X\mapsto K_*^{\rm geo}(X;\mathcal{A})$ is a generalized homology theory on the category of locally finite CW-complexes. In particular, for any closed subspace $Y\subseteq X$ there is a six term exact sequence
\begin{center}
$\begin{CD}
K_*^{\rm geo}(Y;\mathcal{A}) @>>> K_*^{\rm geo}(X;\mathcal{A}) @>>> K_*^{\rm geo}(X,Y;\mathcal{A}) \\
@AAA @.  @VVV  \\
K_{*-1}^{\rm geo}(X,Y;\mathcal{A}) @<<<K_{*-1}^{\rm geo}(X;\mathcal{A}) @<<< K_{*-1}^{\rm geo}(Y;\mathcal{A})  \\
\end{CD}$
\end{center}
\end{theorem}

Here the group $K_*^{\rm geo}(X,Y;\mathcal{A})$ is a relative group. It was studied for $C^*$-algebras in \cite[Section 2.1]{Wal}. We use these groups only for $\mathcal{A}=\mathbb{C}$; see specifics below in Subsection \ref{relagesubsec}. The proof of Theorem \ref{sameasinappbutfre} goes as in \cite[Appendix]{DG}.

\subsection{Relative geometric $K$-homology}
\label{relagesubsec}

We recall the central points in the construction of relative geometric $K$-homology of continuous mappings $h:Y\to X$ and continuous $*$-homomorphisms $\phi:\mathcal{A}_1\to \mathcal{A}_2$ in this subsection. The case that $h$ is an inclusion was considered in \cite{BHS,Wal}. The details of the standard generalization to any continuous map are given in \cite{DGrelI}. The case that $\phi$ is a $*$-homomorphism between $C^*$-algebras was studied in \cite{DeeRZ}. The example we have in mind arises from a group homomorphism $\phi:\Gamma_1\to \Gamma_2$ which induces a continuous mapping $B\phi:B\Gamma_1\to B\Gamma_2$ and, for suitable Fr\'echet algebra completions $\mathcal{A}_i\supseteq \mathbb{C}[\Gamma_i]$, a continuous $*$-homomorphism $\phi:\mathcal{A}_1\to \mathcal{A}_2$. Given mappings $h_i:Y_i\to X_i$, $i=1,2$, $f:Y_1\to Y_2$ and $g:X_1\to X_2$ such that $h_2\circ f=g\circ h_1$ we write 
$$(f,g):[h_1:Y_1\to X_1]\to [h_2:Y_2\to X_2].$$ 

\subsubsection{Relative $K$-homology of continuous mappings}

The following definition can be found in \cite[Definition 3.1]{DGrelI}. The reader should recall from \cite[Definition 3.2]{DGrelI} that a regular domain in a closed manifold is an open subset with smooth boundary.

\begin{define}
Let $h:Y\to X$ be a continuous mapping of topological spaces. A relative geometric cycle for $h$ is a collection $(W,E,(f,g))$ where 
\begin{enumerate}
\item $W$ is a spin$^c$-manifold with boundary;
\item $E\to W$ is a vector bundle;
\item $(f,g):[i:\partial W\to W]\to [h:Y\to X]$ is a continuous mapping, where $i$ denotes the boundary inclusion.
\end{enumerate}
If $((Z,W),E,(f,g))$ is a collection, with $Z$ being a spin$^c$-manifold with boundary and $W\subseteq \partial Z$ being a regular domain, such that $(Z,E,f)$ is a geometric cycle with boundary for $X$ and $(\partial Z\setminus W^\circ ,E|_{\partial Z\setminus W^\circ},g)$ is a geometric cycle with boundary for $Y$, we say that $((Z,W),E,(f,g))$ is a relative geometric cycle with boundary. We write $\partial ((Z,W),E,(f,g)):=(W,E|_{W},(f|_{W},g|_{\partial W}))$.
\end{define}

The set of isomorphism classes of relative cycles for $h$ forms an abelian semigroup under disjoint union. This set is equipped with a Baum-Douglas relation: the equivalence relation generated by the equivalence relations disjoint union/direct sum, bordism and vector bundle modification. We let $K_*^{\rm geo}(h)$, or sometimes $K_*^{\rm geo}(h:Y\to X)$, denote the set of equivalence classes. This set forms a $\mathbb{Z}/2$-graded abelian group where the grading comes from the dimension of (the connected components of the manifold in) the cycle modulo two and the group operation is given by disjoint union. For details, see \cite[Section 3]{DGrelI}. 

\begin{theorem} (Theorem 3.7 of \cite{DGrelI}) \\
\label{sameasinaabove}
Let $h:Y\to X$ be a continuous mapping. The abelian group $K_*^{\rm geo}(h)$ depends functorially on $[h:Y\to X]$ and fits into a long 2-periodic exact sequence
$$\cdots\xrightarrow{\delta} K_*^{\rm geo}(Y)\xrightarrow{h_*} K_*^{\rm geo}(X)\xrightarrow{r} K_*^{\rm geo}(h)\xrightarrow{\delta} K_{*-1}^{\rm geo}(Y)\xrightarrow{h_*} \cdots$$
Here $r:K_*^{\rm geo}(X)\to K_*^{\rm geo}(h)$ maps $(M,E,f)\mapsto (M,E,(f,\emptyset))$ and $\delta:K_*^{\rm geo}(h)\to K_{*-1}^{\rm geo}(Y)$ maps $(W,E,(f,g))\mapsto (\partial W,E|_{\partial W},g)$.
\end{theorem}

We refer the reader to the proof in \cite{BHS} where the case of $h$ being an inclusion was considered. As mentioned in \cite{DGrelI}, the proofs generalize without major changes. It was proven in \cite{DGrelI} that if $X$ and $Y$ are locally finite $CW$-complexes, $K_*^{\rm geo}(h:Y\to X)$ is isomorphic to the $K$-theory of the mapping cone of the functorially associated $*$-homomorphism $h_L:C^*_L(Y)\to C^*_L(X)$ at the level of localization algebras. The proof is a straight forward application of the five lemma.

\subsubsection{Geometric models of relative $K$-theory}
\label{remarkongengeomodel}

There are several models for relative $K$-theory of a $*$-homomorphism $\phi:\mathcal{A}_1\to \mathcal{A}_2$. The model we discuss is geometric and comes from a particular relative version of geometric $K$-homology of a point; it is well suited to the geometric assembly map discussed above. 

\begin{define}
Let $X$ be a topological space and $\phi:\mathcal{A}_1\to \mathcal{A}_2$ a continuous $*$-homomorphism of Fr\'echet algebras. A cycle for $K_*^{\rm geo}(X;\phi)$ is a collection $(W, (E_{\mathcal{A}_2}, F_{\mathcal{A}_1}, \alpha),f)$ where
\begin{enumerate}
\item $W$ is a smooth, compact spin$^c$-manifold with boundary;
\item $E_{\mathcal{A}_2}$ is a $\mathcal{A}_2$-bundle over $W$;
\item $F_{\mathcal{A}_1}$ is a $\mathcal{A}_1$-bundle over $\partial W$;
\item $\alpha: E_{\mathcal{A}_2}|_{\partial W} \rightarrow F_{\mathcal{A}_1}\otimes_{\phi} \mathcal{A}_2$ is an isomorphism of $\mathcal{A}_2$-bundles;
\item $f:W\to X$ is a continuous mapping.
\end{enumerate}
\end{define} 

The set of isomorphism classes of cycles for $K_*^{\rm geo}(X;\phi)$ forms an abelian semigroup under disjoint union. Again, this set comes equipped with a Baum-Douglas relation (generated by disjoint union/direct sum, bordism and vector bundle modification) and the set of equivalence classes is denoted by $K_*^{\rm geo}(X;\phi)$. The set $K_*^{\rm geo}(X;\phi)$ is a $\mathbb{Z}/2$-graded abelian group. The group $K_*^{\rm geo}(X;\phi)$ depends covariantly on the topological space $X$ in the obvious way and the $*$-homomorphism $\phi$ in the following way. If $\phi':\mathcal{A}'_1\to \mathcal{A}'_2$ is another $*$-homomorphism and $\alpha_i:\mathcal{A}_i\to \mathcal{A}'_i$, $i=1,2$, are $*$-homomorphisms such that $\alpha_2\circ \phi=\phi'\circ \alpha_1$ then there is an induced mapping $(\alpha_1,\alpha_2)_*: K_*^{\rm geo}(X;\phi)\to K_*^{\rm geo}(X;\phi')$. We are as mentioned above mainly interested in the case that $X$ is a point. The next theorem summarizes the main properties of $K_*^{\rm geo}(X;\phi)$, the reader is referred to \cite[Theorem 4.21]{DeeRZ} for a proof of the result.

\begin{theorem}
\label{lesforphi}
Let $X$ be a topological space and $\phi:\mathcal{A}_1\to \mathcal{A}_2$ a continuous $*$-homomorphism of Fr\'echet algebras. The mapping $\phi_*:K_*^{\rm geo}(X;\mathcal{A}_1)\to K_*^{\rm geo}(X;\mathcal{A}_2)$ fits into a long 2-periodic exact sequence
$$\cdots\xrightarrow{\delta} K_*^{\rm geo}(X;\mathcal{A}_1)\xrightarrow{\phi_*} K_*^{\rm geo}(X;\mathcal{A}_2)\xrightarrow{r} K_*^{\rm geo}(X;\phi)\xrightarrow{\delta} K_{*-1}^{\rm geo}(X;\mathcal{A}_1)\xrightarrow{\phi_*} \cdots$$
Here $r:K_*^{\rm geo}(X;\mathcal{A}_2)\to K_*^{\rm geo}(X;\phi)$ maps $(M,\mathcal{E}_\mathcal{A},f)\mapsto (M,(\mathcal{E}_\mathcal{A},\emptyset,\emptyset),f)$ and $\delta:K_*^{\rm geo}(X;\phi)\to K_{*-1}^{\rm geo}(X;\mathcal{A}_1)$ maps $(W, (E_{\mathcal{A}_2}, F_{\mathcal{A}_1}, \alpha),f)\mapsto (\partial W,F_{\mathcal{A}_1},f |_{\partial W})$.
\end{theorem}

The reader should note that in the definition of $r$, the domain of definition is $K_*^{\rm geo}(X;\mathcal{A}_2)$ where the manifolds appearing in the cycles are closed and their boundaries are empty. The empty set has no bundles over it, and the bundle data on the boundary can therefore only be the empty set.

\begin{remark}
We note the following consequence of Theorem \ref{lesforphi} and Theorem \ref{closedandkhom}. If $\mathcal{A}_1$ and $\mathcal{A}_2$ are closed under holomorphic functional calculus in Banach algebras to which $\phi$ extends, a five lemma argument implies that there exists an abstract isomorphism $K_*^{\mathrm{geo}}(pt; \phi)\cong K_*(SC_\phi)$ where 
$$C_{\phi}:=\left\{a\oplus b\in C_0((0,1],\mathcal{A}_2)\oplus \mathcal{A}_1: \; a(1)=\phi(b)\right\}.$$
\end{remark}

\subsection{Relative assembly and the strong relative Novikov property}
\label{subseonrelassforfre}

The assembly mapping for free actions $\mu:K_*(B\Gamma)\to K_*(C^*_\epsilon(\Gamma))$ (for some $C^*$-completion $C^*_\epsilon(\Gamma)\supseteq \mathbb{C}[\Gamma]$) has a geometric description studied in \cite{DG,land}. We assume our model of $B\Gamma$ is a locally finite $CW$-complex. This fact makes a definition of free assembly possible also for Fr\'echet algebra completions $\mathcal{A}$ of $\mathbb{C}[\Gamma]$, a fact used in \cite{DGIII}. The main object used in free assembly is the Mishchenko bundle 
\begin{equation}
\label{michscdef}
\mathcal{L}_{\mathcal{A}}:=E\Gamma\times_\Gamma \mathcal{A}\to B\Gamma.
\end{equation}
On any cell in $B\Gamma$, $\mathcal{L}_\mathcal{A}$ is trivializable to a rank one trivial $\mathcal{A}$-bundle with locally constant transition functions. In particular, for any continuous mapping $f:M\to B\Gamma$ from a smooth manifold $M$, the $\mathcal{A}$-bundle $f^*\mathcal{L}_\mathcal{A}\to M$ admits a canonical $\mathcal{A}$-linear unitary flat connection.

\begin{define}
The free assembly mapping $\mu_{\rm geo}^\mathcal{A}:K_*^{\rm geo}(B\Gamma)\to K_*^{\rm geo}(pt;\mathcal{A})$ is defined at the level of cycles by 
$$\mu_{\rm geo}^\mathcal{A}(M,E,f):=(M,E\otimes f^*\mathcal{L}_\mathcal{A}).$$ 
\end{define}

It is clear from the definition that $\mu_{\rm geo}^\mathcal{A}$ is well defined from classes to classes as it preserves the Baum-Douglas relations. By Theorem \ref{closedandkhom}, the left hand side $K_*^{\rm geo}(B\Gamma)$ is isomorphic to the compactly supported $K$-homology of $B\Gamma$, assuming we have chose $B\Gamma$ to be a locally finite $CW$-complex. Moreover, it was shown in \cite{land} that under the isomorphisms of Theorem \ref{closedandkhom}, $\mu_{\rm geo}^{C^*_{\bf red}}$ coincides with Kasparov's definition of the free assembly mapping. We now turn to the relative case. The relative free assembly mapping relies on the following lemma, the reader can find its proof in \cite[Proposition 3.15]{DGrelI}.

\begin{lemma}
\label{isomopfmish}
Let $\phi:\Gamma_1\to \Gamma_2$ be a group homomorphism inducing a continuous mapping $B\phi:B\Gamma_1\to B\Gamma_2$ and a continuous $*$-homomorphism $\phi_{\mathcal{A}}:\mathcal{A}_1\to \mathcal{A}_2$. Then, there is a canonical isomorphism of $\mathcal{A}_2$-bundles on $B\Gamma_1$:
$$\alpha_0:(B\phi)^*\mathcal{L}_{\mathcal{A}_2}\xrightarrow{\sim} \mathcal{L}_{\mathcal{A}_1}\otimes_{\phi_\mathcal{A}} \mathcal{A}_2.$$
\end{lemma}

The reader should note that $\mathcal{L}_{\mathcal{A}_1}$ is a locally trivial rank one $\mathcal{A}_1$-bundle so there is no need for completion in the fibre wise algebraic tensor product  $\mathcal{L}_{\mathcal{A}_1}\otimes_{\phi_\mathcal{A}} \mathcal{A}_2$.

\begin{theorem}
\label{relativeassthm}
The free relative assembly map $\mu_{\rm geo}^{\phi_\mathcal{A}}:K_*^{\rm geo}(B\phi)\to K_*^{\rm geo}(pt;\phi_\mathcal{A})$, defined by  
$$(W,E,(f,g))\mapsto (W,(E\otimes f^*\mathcal{L}_{\mathcal{A}_2},E|_{\partial W}\otimes g^*\mathcal{L}_{\mathcal{A}_1},\mathrm{id}_{E|_{\partial W}}\otimes g^*(\alpha_0))),$$ 
is well defined and fits into a commuting diagram with long 2-periodic exact rows: 
\small
\begin{equation}
\label{yeatanothercommuting}
 \minCDarrowwidth15pt
\begin{CD}
@>>> K_*^{\textnormal{geo}}(B\Gamma_1) @>B\phi_*>> K_*^{\textnormal{geo}}(B\Gamma_2) @>r>> K_*^{\textnormal{geo}}(B\phi) @>\delta >> K_{* -1}^{\textnormal{geo}}(B\Gamma_1) @>>> \\
@. @VV \mu_{\rm geo}^{\mathcal{A}_1} V @VV\mu_{\rm geo}^{\mathcal{A}_2} V  @VV\mu_{\rm geo}^{\phi_\mathcal{A}}  V  @VV\mu_{\rm geo}^{\mathcal{A}_1}  V \\
@>>> K_*^{\textnormal{geo}}(pt;\mathcal{A}_1) @>\phi_\mathcal{A} >> K_*^{\textnormal{geo}}(pt;\mathcal{A}_2) @>r>> K_*^{\textnormal{geo}}(pt;\phi_\mathcal{A}) @>\delta>> K_{* -1}^{\textnormal{geo}}(pt;\mathcal{A}_1) @>>> \\
\end{CD}
\end{equation}
\normalsize
\end{theorem}

The proof of Theorem \ref{relativeassthm} goes mutatis mutandis to \cite[Theorem 3.18]{DGrelI}.

\begin{remark}
The functoriality properties of the involved groups imply that whenever $(f,g):[h:Y\to X]\to[B\phi:B\Gamma_1\to B\Gamma_2]$, there is an associated free assembly $\mu_{\rm geo}:=\mu_{\rm geo}^{\phi_\mathcal{A}}\circ (f,g)_*:K_*^{\rm geo}(h)\to K_*^{\rm geo}(pt;\phi_\mathcal{A})$ which fits into a commuting diagram as \eqref{yeatanothercommuting}.
\end{remark}

A discrete group $\Gamma$ is said to satisfy the strong Novikov conjecture if the free assembly mapping $\mu_{\rm geo}^{C^*_{\bf max}}:K_*^{\rm geo}(B\Gamma)\to K_*^{\rm geo}(pt;C^*(\Gamma))$ is a rational injection and the reduced strong Novikov conjecture if the reduced free assembly mapping $\mu_{\rm geo}^{C^*_{\bf red}}:K_*^{\rm geo}(B\Gamma)\to K_*^{\rm geo}(pt;C^*_{\bf red}(\Gamma))$ is a rational injection. The group $\Gamma$ is said to satisfy the the reduced Baum-Connes property if the Baum-Connes assembly mapping $K_*^\Gamma(\underline{\mathcal{E}}\Gamma)\to K_*(C^*_{\bf red}(\Gamma)$ (for proper actions) is an isomorphism. If $\Gamma$ is torsion-free, Baum-Connes assembly mapping coincides with the assembly mapping for free actions up to the canonical isomorphism  $K_*^\Gamma(\underline{\mathcal{E}}\Gamma)\cong K_*(B\Gamma)$. The strong Novikov conjecture follows from the reduced strong Novikov conjecture. The reduced strong Novikov conjecture follows from the reduced Baum-Connes property of $\Gamma$. 

\begin{define}
Consider a group homomorphism $\phi:\Gamma_1\to \Gamma_2$ inducing a continuous $*$-homomorphism $\phi_{\mathcal{A}}:\mathcal{A}_1\to \mathcal{A}_2$. 
\begin{enumerate}
\item Assume that $\Gamma_1$ and $\Gamma_2$ are torsion free. If $\mu_{\rm geo}^{\phi_\mathcal{A}}$ is an isomorphism, we say that $\phi_{\mathcal{A}}:\mathcal{A}_1\to \mathcal{A}_2$ has the Baum-Connes property.
\item If $\mu_{\rm geo}^{\phi_\mathcal{A}}$ is a rational injection, we say that $\phi_{\mathcal{A}}:\mathcal{A}_1\to \mathcal{A}_2$ has the strong relative Novikov property.
\end{enumerate}
In (2), if $\mathcal{A}_i=C^*_\epsilon(\Gamma_i)$, $i=1,2$, are $C^*$-algebra completions we simply say that $\phi$ has the $\epsilon$-strong relative Novikov property.
\end{define}

\begin{remark}
\label{strongremark}
By an argument of Weinberger \cite{weinnov}, the relative higher signatures of manifolds with boundaries are homotopy invariants if the fundamental group of the manifold with boundary satisfies the Novikov conjecture and the boundary satisfies the Borel conjecture. This justifies the terminology ``strong relative Novikov property" as we make no claims on its veracity. It was proven in \cite[Theorem 4.18]{DGrelI} that the strong relative Novikov property implies such homotopy invariance of relative higher signatures of manifolds with boundaries as in \cite[Section 12]{weinnov} and \cite[Section 4.9]{Lott}.  
\end{remark}

\begin{remark}
The absolute Baum-Connes conjecture for a group $\Gamma$ implies the strong Novikov conjecture. The Baum-Connes conjecture for two torsion-free groups $\Gamma_1$ and $\Gamma_2$, and the five lemma, implies the relative Baum-Connes conjecture for any homomorphism $\phi:\Gamma_1\to \Gamma_2$. However, to the authors knowledge, there is no general way to deduce the strong relative Novikov property from the Baum-Connes property if there is torsion in $\Gamma_1$ and $\Gamma_2$. 
\end{remark}

\subsection{De Rham homology and Chern characters}
\label{subseconder}

We now turn to the noncommutative de Rham theory for Fr\'echet algebras and Chern characters on geometric models, both in the absolute setting. Below in Section \ref{chernonk} we return to the relative setting where the goal is to complete the diagram \eqref{yeatanothercommuting} (see page \pageref{yeatanothercommuting}) with an additional level reached through Chern characters describing the situation in homology. We follow the presentation of \cite{DGIII}.

\subsubsection{Noncommutative de Rham theory}

We fix a unital algebra $\mathcal{A}$. We define $\Omega^0(\mathcal{A}):=\mathcal{A}$ and $\Omega_1(\mathcal{A}):= \mathcal{A}\otimes (\mathcal{A}/\field{C}1_\mathcal{A})$. The algebra $\mathcal{A}$ carries a universal derivation $\mathrm{d}_\mathcal{A}:\Omega_0(\mathcal{A})=\mathcal{A}\to \Omega_1(\mathcal{A})$, $a\mapsto \mathrm{d}a=1\otimes a$. Here $\Omega_1(\mathcal{A})$ is equipped with the bimodule structure making $\mathrm{d}_{\mathcal{A}}$ into a derivation. We set $\Omega_k(\mathcal{A}):=\Omega_1(\mathcal{A})^{\otimes_\mathcal{A} k}=\mathcal{A}\otimes (\mathcal{A}/\field{C}1_\mathcal{A})^{\otimes_\field{C} k}$ for $k>0$. One constructs the universal $\Z$-graded differential algebra of $\mathcal{A}$ as $\Omega_*(\mathcal{A}):=\prod_{k=0}^\infty \Omega_k(\mathcal{A})$. The Leibniz rule extends $\mathrm{d}_{\mathcal{A}}$ to a derivation on $\Omega_*(\mathcal{A})$; we use the same notation $\mathrm{d}_\mathcal{A}$ for the derivation on $\Omega_*(\mathcal{A})$. The abelianized complex is defined as $\Omega_*^{\rm ab}(\mathcal{A}):=\Omega_*(\mathcal{A})/[\Omega_*(\mathcal{A}),\Omega_*(\mathcal{A})]$; all commutators are graded. For a unital Fr\'echet  $*$-algebra $\mathcal{A}$, we can replace all algebraic tensor products by projective tensor products arriving at a $\Z$-graded differential Fr\'echet algebras $\hat{\Omega}_*(\mathcal{A})$ and we set $\hat{\Omega}_*^{\rm ab}(\mathcal{A}):=\hat{\Omega}_*(\mathcal{A})/\overline{[\hat{\Omega}_*(\mathcal{A}),\hat{\Omega}_*(\mathcal{A})]}$. For additional details on topological homology, see \cite{helembook,jotaylor}.

We let $\tilde{\mathcal{A}}$ denote the unitalization of an algebra $\mathcal{A}$. Unitalization defines a unital Fr\'echet algebra from a Fr\'echet algebra.

\begin{define}
If $\mathcal{A}$ is an algebra, we define its de Rham homology $H_*^{\rm dR}(\mathcal{A}):=H_*(\Omega_*^{\rm ab}(\tilde{\mathcal{A}}))$. If $\mathcal{A}$ is a Fr\'echet algebra, we define its topological de Rham homology $\hat{H}_*^{\rm dR}(\mathcal{A})$ as the topological homology of $\hat{\Omega}_*^{\rm ab}(\tilde{\mathcal{A}})$ (that is, closed cycles modulo the closure of the exact cycles).
\end{define}

The de Rham homology is related to cyclic homology. We denote the Hochschild complex of an algebra $\mathcal{A}$ by $\mathcal{C}_*(\mathcal{A})$ and the cyclic complex by $\mathcal{C}_*^\lambda(\mathcal{A})$. We denote the boundary operator on both of these complexes by $b$. The Hochschild homology will be denoted by $HH_*(\mathcal{A})$ and the cyclic homology by $HC_*(\mathcal{A})$. For details, see \cite{cuncycenc,Kaast}. The  $S$-operator on cyclic homology is constructed in for instance \cite[Page 11]{cuncycenc}. The next result is from \cite[Theorem 2.15]{Kaast}, it can also be found as \cite[Proposition 3.2]{DGIII}.

\begin{prop}
\label{hdrandhcyc}
The map
\begin{equation}
\label{omegatocmodb}
\Omega_k^{\rm ab}(\tilde{\mathcal{A}})\ni a_0\mathrm{d} a_1\cdots \mathrm{d} a_k\mapsto a_0\otimes a_1\otimes \cdots a_k\in \mathcal{C}_k^\lambda(\mathcal{A})/\mathrm{Im}(b)
\end{equation}
induces an isomorphism 
$$H_*^{\rm dR}(\mathcal{A})\cong \mathrm{Im}(S:HC_{*+2}(\mathcal{A})\to HC_*(\mathcal{A}))=\ker(B:HC_*(\mathcal{A})\to HH_{*+1}(\mathcal{A})).$$
\end{prop}

We return to the setting of Fr\'echet  $*$-algebra completions $\mathcal{A}$ of $\C[\Gamma]$ for a discrete finitely generated group $\Gamma$. The de Rham homology of $\C[\Gamma]$ contains the homology of $B\Gamma$. We denote the unitalization of the group algebra by $\tilde{\C}[\Gamma]$. There is a differential graded map 
$$j_{\Omega,\mathcal{A}}:\Omega_*^{\rm ab}(\tilde{\C}[\Gamma])\to \hat{\Omega}_*^{\rm ab}(\tilde{\mathcal{A}}).$$
The unit in $\Gamma$ will be denoted by $e$ and the formal unit in a unitalization by $1$. The sub complex $\Omega^{<e>}_*(\C[\Gamma])\subseteq \Omega_*^{\rm ab}(\tilde{\C}[\Gamma])$ is defined as the linear span of the set
\begin{align}
\label{locdaefe}
\cup_{k\in \field{N}}\{g_0\mathrm{d}&g_1\mathrm{d}g_2\cdots \mathrm{d}g_k: \; g_0g_1\cdots g_k=e\}\quad\mbox{and}\\ 
\nonumber
&\cup_{k\in \field{N}}\{1\cdot \mathrm{d}g_1\mathrm{d}g_2\cdots \mathrm{d}g_k: \; g_1\cdots g_k=e\},
\end{align}
The homology of $\Omega^{<e>}_*(\C[\Gamma])$ is denoted by $H^{<e>}_*(\C[\Gamma])$. The next result follows from Proposition \ref{hdrandhcyc} and \cite[Section 2.21-2.26]{Kaast}, building on the results from \cite{burgher}. For a group element $\gamma\in \Gamma$, we let $<\gamma>\subseteq \Gamma$ denote its conjugacy class. To shorten notation, we sometimes write $R_*$ for $HC_*(\C)$. The mapping $\beta_\Gamma:H^{<e>}_*(\C[\Gamma])\xrightarrow{\sim}H_*(B\Gamma)\otimes HC_{*}(\C)$ is defined as the composition:
\begin{align*}
H^{<e>}_*(\C[\Gamma])&\to HC_*(\C[\Gamma])\cong \begin{matrix}
H_*(B\Gamma)\otimes R_*\\
\bigoplus\\
\bigoplus_{<\gamma>\in C'\setminus \{<e>\}}H_*(BN_\gamma)\otimes R_*\\
\bigoplus\\ 
\bigoplus_{<\gamma>\in C''}H_*(BN_\gamma)\end{matrix}\to H_*(B\Gamma)\otimes HC_*(\C).
\end{align*}
The last mapping is the projection. The last isomorphism is the Burghelea isomorphism (see \cite{burgher}). The set $C'$ denotes the set of conjugacy classes of finite order elements, $C''$ the set of conjugacy classes of infinite order elements. The group $\Gamma_\gamma$ denotes the centralizer of $\gamma$ and $N_\gamma:=\Gamma_\gamma/\gamma^\Z$.

\begin{prop}
\label{hstarhstar}
The mapping $\beta_\Gamma:H^{<e>}_*(\C[\Gamma])\xrightarrow{\sim}H_*(B\Gamma)\otimes HC_{*}(\C)$ is an isomorphism
\end{prop}

\subsubsection{Admissible completions}

The localized part $H^{<e>}_*(\C[\Gamma])$ of the de Rham homology of $\mathbb{C}[\Gamma]$ relates to the group homology via $\beta_\Gamma$. As such, there is a mapping $\beta^*_\Gamma: H^*(B\Gamma)\otimes HC^{*}(\C)\to \mathrm{Hom}(H^{<e>}_*(\C[\Gamma]),\mathbb{C})$. We embedd $H^*(B\Gamma)=H^*(B\Gamma)\otimes HC^0(\mathbb{C})\hookrightarrow H^*(B\Gamma)\otimes HC^{*}(\C)$ arriving at a mapping $\hat{\beta}_\Gamma:H^*(B\Gamma)\to \mathrm{Hom}(H^{\rm dR}_*(\C[\Gamma]),\mathbb{C})$ by pre-composing with the projection mapping $H^{\rm dR}_*(\C[\Gamma])\to H^{<e>}_*(\C[\Gamma])$. 

To describe this mapping more concretely, we let $\mathcal{C}^*(\Gamma)$ denote the complex of group cochains on $\Gamma$. For any model of $B\Gamma$, $H^*(\mathcal{C}^*(\Gamma))\cong H^*(B\Gamma)$. We often identify the two cohomology groups. For a group cocycle $c\in \mathcal{C}^k(\Gamma)$, the class $\hat{\beta}_\Gamma[c]$ is represented by $\hat{c}:\Omega_k^{\rm ab}(\tilde{\C}[\Gamma])\to \C$ where 
\begin{equation}
\label{thereisahat}
\hat{c}(\gamma_0\mathrm{d}\gamma_1\cdots \mathrm{d}\gamma_p):=
\begin{cases}
c(\gamma_1,\gamma_1\gamma_2,\ldots,\gamma_1\gamma_2\cdots\gamma_p), \;&\gamma_0\in \Gamma, \mbox{  and   }\gamma_0\gamma_1\cdots\gamma_p=e,\\
0,&  \gamma_0=1, \mbox{  or   }\gamma_0\gamma_1\cdots\gamma_p\neq e, \end{cases}
\end{equation}
Recall our convention to denote the adjointed unit by $1$ and the unit in $\Gamma$ by $e$. See more in \cite[Equation (20), page 209]{Lott} or \cite[Page 377]{cmnovikov}. In this way we arrive at a pairing 
$$H^*(B\Gamma)\times H^{\rm dR}_*(\C[\Gamma])\to \C.$$
For a general Fr\'echet algebra completion $\mathcal{A}$ of $\C[\Gamma]$, we can define a localized part $\hat{H}^{<e>}_*(\mathcal{A})$ of the de Rham homology of $\mathcal{A}$ by taking the topological homology of the closure of $\Omega^{<e>}_*(\C[\Gamma])$ in $\hat{\Omega}_*^{\rm ab}(\tilde{\mathcal{A}})$ as in \cite[Section 4.1]{DGIII}. A technical problem arising is that the localized part need not be a direct summand in $\hat{H}^{\rm dR}_*(\mathcal{A})$. The problem of interest is which group cocycles $c\in \mathcal{C}^*(\Gamma)$ extends to mappings $\hat{c}:\hat{\Omega}_k^{\rm ab}(\tilde{\mathcal{A}})\to \C$? An answer to that question allows one to study the strong Novikov property as in \cite{cmnovikov}, we recall this argument below in Proposition \ref{propcompass}.2. To set some notations, let
$$\mathcal{C}^*_\mathcal{A}(\Gamma):=\{c\in \mathcal{C}^*(\Gamma): \, \hat{c} \mbox{   extends to a continuous mapping    }\hat{\Omega}_k^{\rm ab}(\tilde{\mathcal{A}})\to \C\}.$$
This is a sub-complex of $\mathcal{C}^*(\Gamma)$. We set $H^*_\mathcal{A}(\Gamma):=H^*(\mathcal{C}^*_\mathcal{A}(\Gamma))$. Compare to the notation and techniques in \cite{jiora}.

\begin{define}
\label{admissibleextdef}
We say that a Fr\'echet algebra completion $\mathcal{A}$ of a group algebra $\C[\Gamma]$ is admissible if the mapping $H^*_\mathcal{A}(\Gamma) \to H^*(B\Gamma)$, induced by the inclusion $\mathcal{C}^*_\mathcal{A}(\Gamma)\to \mathcal{C}^*(\Gamma)$, is a surjection. 

We say that $\mathcal{A}$ is admissible for a $C^*$-completion $C^*_\epsilon(\Gamma)$ if $\mathcal{A}$ is admissible and the inclusion $\C[\Gamma]\hookrightarrow C^*_\epsilon(\Gamma)$ extends by continuity to $\mathcal{A}$ inducing an isomorphism $K_*(\mathcal{A})\to K_*( C^*_\epsilon(\Gamma))$. 
\end{define}

We now restrict our attention to the completions arising from a length function and provide examples of admissible completions.

\begin{ex}
We first consider the Fr\'echet algebras $H^{1,\infty}_L(\Gamma)$ from Example \ref{ageoneess}. In this case, $c\in \mathcal{C}^k_{H^{1,\infty}_L(\Gamma)}(\Gamma)$ if and only if there is an $s>0$ and a constant $C$ such that 
$$|c(\gamma_1,\gamma_2,\ldots, \gamma_k)|\leq C\left(1+\sum_{i=1}^k L(g_i)\right)^s.$$
That is, $c\in \mathcal{C}^k_{H^{1,\infty}_L(\Gamma)}(\Gamma)$ exactly when $c$ has polynomial growth. On \cite[bottom of Page 384]{cmnovikov}, the property that any cohomology class on a group $\Gamma$ is representable by a cocycle of polynomial growth is called (PC). See also below in Definition \ref{pcandpolbdd}. In particular, (PC) is equivalent to $H^{1,\infty}_L(\Gamma)$ being admissible. The completion $H^{1,\infty}_L(\Gamma)$ is admissible for a $C^*$-completion $C^*_\epsilon(\Gamma)$ if and only if (PC) holds and $\ell^1(\Gamma)\to C^*_\epsilon(\Gamma)$ induces an isomorphism on $K$-theory (see Example \ref{ageoneess}). 

Recall the notion of property (RD) and the definition of $H^\infty_L(\Gamma)$ from Example \ref{firstappofrd}. If $(\Gamma,L)$ has property (RD), a similar description of $\mathcal{C}^*_{H^{\infty}_L(\Gamma)}(\Gamma)$ is possible but slightly more complicated. By functoriality, $\mathcal{C}^*_{H^{\infty}_L(\Gamma)}(\Gamma)\subseteq \mathcal{C}^*_{H^{1,\infty}_L(\Gamma)}(\Gamma)$ and $H^{1,\infty}_L(\Gamma)$ is admissible if $H^{\infty}_L(\Gamma)$ is. In fact, \cite[Proof of Proposition 6.5]{cmnovikov} shows that $H^{1,\infty}_L(\Gamma)$ is admissible if and only if $H^{\infty}_L(\Gamma)$ is. For a group $(\Gamma,L)$ with property (RD) we conclude that $H^{\infty}_L(\Gamma)$ is admissible for $C^*_{\bf red}(\Gamma)$ if and only if (PC) holds.
\end{ex} 

We sum up the property above and a property that will come to use in the relative setting into the following definition.

\begin{define}
\label{pcandpolbdd}
Let $\Gamma$ denote a finitely generated group with a length function $L$. 
\begin{itemize}
\item We say that $(\Gamma,L)$ has (PC) if $H^{1,\infty}_L(\Gamma)$ is admissible, i.e. the inclusion $\mathcal{C}^*_{H^{1,\infty}_L(\Gamma)}(\Gamma)\to \mathcal{C}^*(\Gamma)$ induces a surjection on cohomology.
\item We say that $(\Gamma,L)$ has polynomially bounded cohomology in $\C$ if the inclusion $\mathcal{C}^*_{H^{1,\infty}_L(\Gamma)}(\Gamma)\to \mathcal{C}^*(\Gamma)$ induces an isomorphism $H^*_{H^{1,\infty}_L(\Gamma)}(\Gamma)\to H^*(B\Gamma)$ on cohomology.
\end{itemize}
If a group $\Gamma$ admits a length function such that $(\Gamma,L)$ has (PC), or having polynomially bounded cohomology in $\C$, we fix such a length function and simply say that $\Gamma$ has (PC) or polynomially bounded cohomology in $\C$.
\end{define}

\begin{ex}
\label{hyperelleone}
We now assume that $\Gamma$ is hyperbolic and consider the Banach $*$-algebra $\ell^1(\Gamma)$. We have that $\mathcal{C}^k_{\ell^{1}(\Gamma)}(\Gamma)$ consists of the bounded cocycles $c\in \mathcal{C}^k(\Gamma)$, i.e. those for which there is a constant $C$ such that 
$$|c(\gamma_1,\gamma_2,\ldots, \gamma_k)|\leq C.$$
The cohomology group $H^*_{\ell^1}(\Gamma)$ coincides with the bounded cohomology of $\Gamma$. For hyperbolic groups, the next proposition shows that $\ell^1(\Gamma)$, $H^{1,\infty}_L(\Gamma)$ and $H^{\infty}_L(\Gamma)$ are all admissible completions.

\begin{prop}
\label{hyperiso}
If $\Gamma$ is a hyperbolic group, $H^*_{\ell^1}(\Gamma)=H^*(B\Gamma)$.
\end{prop}

\begin{proof}
We say that a Banach space $V$ is a bounded $\Gamma$-module if it admits an action of $\Gamma$ for which there is a constant $C$ with $\|\gamma\|_{V\to V}\leq C$ for all $\gamma\in \Gamma$. We let $H^*_b(\Gamma,V)$ denote the bounded cohomology of $\Gamma$ with coefficients in $V$. By \cite[Theorem 11]{minstra}, the mapping $H^*_{b}(\Gamma,V)\to H^*(\Gamma,V)$ is surjective for any bounded $\Gamma$-module $V$ when $\Gamma$ is hyperbolic. This fact and using that hyperbolic groups are $FP^\infty$ (see \cite[Theorem 10.2.6]{epsteinetal} and \cite{alosopa}), \cite[Theorem 2]{jiora} implies that $H^*_{b}(\Gamma,V)\to H^*(\Gamma,V)$ is an isomorphism for any bounded $\Gamma$-module $V$. By combining the two results we arrive at the identity $H^*_{\ell^1}(\Gamma)=H^*_b(\Gamma,\C)\cong H^*(\Gamma,\C)=H^*(B\Gamma)$ for any hyperbolic group $\Gamma$. 
\end{proof}
\end{ex}

\subsubsection{Absolute Chern characters} 
\label{absolutsubse}
On $K$-theory, Chern characters are straight forward to define. In \cite{Connesbook,Kaast} among other places, there are Chern characters defined from $K$-theory into de Rham homology and cyclic homology. When working with the geometric models, Chern characters are defined as in \cite[Section 4]{DGIII} -- a construction built to fit with higher index theory as in \cite{lottsuper, SchL2,WahlAPSForCstarBun}. 

As before, $\mathcal{A}$ denotes a unital Fr\'echet $*$-algebra. If $\mathcal{E}_\mathcal{A}\to W$ is an $\mathcal{A}$-bundle on a manifold with boundary, an $\mathcal{A}$-linear connection $\nabla_\mathcal{E}$ is an $\mathcal{A}$-linear operator $\nabla_\mathcal{E}:C^\infty(W,\mathcal{E}_{\mathcal{A}})\to C^\infty(W,\mathcal{E}_{\mathcal{A}}\otimes T^*M)$ satisfying the Leibniz rule $\nabla_\mathcal{E}(fa)=\nabla_\mathcal{E}(f)a+f\mathrm{d}_Wa$, $a\in C^\infty(W)$, $f\in C^\infty(W,\mathcal{E}_\mathcal{A})$, where $\mathrm{d}_W$ denotes the exterior derivative. A total connection $\tilde{\nabla}_\mathcal{E}$ lifting $\nabla_\mathcal{E}$ is a linear mapping 
$$\tilde{\nabla}_\mathcal{E}:C^\infty(W,\mathcal{E}_{\mathcal{A}})\to C^\infty(W,\mathcal{E}_{\mathcal{A}}\otimes T^*M\oplus \mathcal{E}_{\mathcal{A}}\otimes_\mathcal{A}\hat{\Omega}_1^{\rm ab}(\tilde{\mathcal{A}}))$$ 
such that 
$$\tilde{\nabla}_\mathcal{E}(fa)=\nabla_\mathcal{E}(f)a+f(\mathrm{d}_W+\mathrm{d}_\mathcal{A})a, \quad a\in C^\infty(W,\tilde{\mathcal{A}}), \;f\in C^\infty(W,\mathcal{E}_\mathcal{A})$$
We note that if $\tilde{\nabla}_\mathcal{E}$ is a total connection lifting $\nabla_\mathcal{E}$, then $\mathrm{d}_\mathcal{E}:=\tilde{\nabla}_\mathcal{E}-\nabla_\mathcal{E}$ is a ``connection in the $\mathcal{A}$-direction" in the sense that, up to a multiplication operator we have the mapping property $\mathrm{d}_\mathcal{E}:C^\infty(W,\mathcal{E}_{\mathcal{A}})\to C^\infty(W,\mathcal{E}_{\mathcal{A}}\otimes_\mathcal{A} \hat{\Omega}_1^{\rm ab}(\tilde{\mathcal{A}}))$. Moreover, $\mathrm{d}_\mathcal{E}$ satisfies the Leibniz rule 
$$\mathrm{d}_\mathcal{E}(fa)=\mathrm{d}_\mathcal{E}(f)a+f\mathrm{d}_\mathcal{A}a,\quad a\in C^\infty(W,\tilde{\mathcal{A}}),\;f\in C^\infty(W,\mathcal{E}_\mathcal{A}).$$

\begin{define}
\label{cycleswitcondef}
A collection $(M, \mathcal{E}_\mathcal{A},f,\tilde{\nabla}_\mathcal{E})$ is called a cycle with total connection for $K_*^{\rm geo}(X; \mathcal{A})$ whenever $(M, \mathcal{E}_\mathcal{A},f)$ is a cycle for $K_*^{\rm geo}(X; \mathcal{A})$ and $\tilde{\nabla}_\mathcal{E}$ is a total connection for $\mathcal{E}_{\mathcal{A}}$. In this case, we say that $\tilde{\nabla}_\mathcal{E}$ is a total connection for the cycle $(M, \mathcal{E}_{\mathcal{A}},f)$.
\end{define}

If $X=pt$, a cycle with total connection is written simply as $(M, \mathcal{E}_\mathcal{A},\tilde{\nabla}_\mathcal{E})$. If $\mathcal{A}=\C$, we just refer to $(M, E,f,\nabla_E)$ as a cycle with connection. 

\begin{ex}
Any cycle $(M, \mathcal{E}_{\mathcal{A}})$ for $K_*^{\rm geo}(pt; \mathcal{A})$ admits a total connection. One example is the total Gra\ss mannian connection $\tilde{\nabla}_\mathcal{E}^{\rm gr}$ that lifts the Gra\ss mannian connection $\nabla_\mathcal{E}^{\rm gr}$. The Gra\ss mannian connection is constructed by choosing an idempotent $p\in C^\infty(M,M_N(\mathcal{A}))$ such that $C^\infty(M,\mathcal{E}_{\mathcal{A}})\cong pC^\infty(M,\mathcal{A}^N)$ as $C^\infty(M,\mathcal{A})$-modules, and defining $\tilde{\nabla}_\mathcal{E}$ as the composition 
\begin{align*}
C^\infty(M,\mathcal{E}_{\mathcal{A}})\cong &pC^\infty(M,\mathcal{A}^N)\hookrightarrow C^\infty(M,\mathcal{A}^N)\\
&\xrightarrow{\mathrm{d}_W+\mathrm{d}_{\mathcal{A}}}C^\infty(M,\mathcal{A}^N\otimes T^*M\oplus \mathcal{A}^N\otimes_\mathcal{A}\hat{\Omega}_1^{\rm ab}(\tilde{\mathcal{A}})))\\
&\xrightarrow{p\oplus p}pC^\infty(M,\mathcal{A}^N\otimes T^*M)\oplus pC^\infty(M,\mathcal{A}^N\otimes_\mathcal{A}\hat{\Omega}_1^{\rm ab}(\tilde{\mathcal{A}})))\cong \\
&\cong C^\infty(M,\mathcal{E}_\mathcal{A}\otimes T^*M\oplus \mathcal{E}_\mathcal{A}\otimes_\mathcal{A}\hat{\Omega}_1^{\rm ab}(\tilde{\mathcal{A}}))).
\end{align*}
\end{ex}

\begin{ex}
\label{coveringexample}
Here we restrict to a Fr\'echet $*$-algebra completion $\mathcal{A}$ of the group algebra $\C[\Gamma]$ of a discrete group $\Gamma$. Consider a Galois covering $\tilde{W}\to W$ of a manifold (possibly with boundary) with covering group $\Gamma$. We consider the associated flat Mishchenko bundle $\mathcal{L}_\mathcal{A}:=\tilde{W}\times_\Gamma \mathcal{A}\to W$. Since $\mathcal{L}_\mathcal{A}$ is flat, there is a canonical hermitean $\mathcal{A}$-linear connection. To define a total connection, we need therefore only specify a connection in the $\mathcal{A}$-direction; that is, a map $\mathrm{d}_\mathcal{L}:C^\infty(W,\mathcal{L}_{\mathcal{A}})\to C^\infty(W,\mathcal{L}_{\mathcal{A}}\otimes_\mathcal{A} \Omega_1^{\rm ab}(\tilde{\mathcal{A}}))$ satisfying the Leibniz rule $\mathrm{d}_\mathcal{L}(fa)=\mathrm{d}_\mathcal{L}(f)a+f\mathrm{d}_\mathcal{A}a$. 

To do so, we follow \cite[Proposition 9]{lottsuper}. Fix $h\in C^\infty_c(\tilde{W},[0,1])$ such that $\sum_{\gamma\in \Gamma} h(x\gamma)=1$ for all $x\in \tilde{W}$. It holds that $C^\infty(W,\mathcal{L}_\mathcal{A})=C^\infty(\tilde{W},\mathcal{A})^\Gamma$ where $C^\infty(\tilde{W},\mathcal{A})$ is equipped with the diagonal action. In particular, any $f\in C^\infty(W,\mathcal{L}_\mathcal{A})$ can be written as a sum $f=\sum_{\gamma\in \Gamma} f_\gamma \lambda_{\gamma}\in C^\infty(\tilde{W},\mathcal{A})$ converging on compacts in the topology of $\mathcal{A}$. Here $\lambda_{\gamma}\in \C[\Gamma]\subseteq \mathcal{A}$ denotes the unitary corresponding to $\gamma\in \Gamma$. The invariance of $f$ ensures that $\gamma_1^*f_{\gamma_2}=f_{\gamma_1^{-1}\gamma_2}$. Define 
\begin{align*}
\mathrm{d}_{\mathcal{L},h}:&C^\infty(W,\mathcal{L}_\mathcal{A})=C^\infty(\tilde{W},\mathcal{A})^\Gamma\to C^\infty(\tilde{W},\Omega_1^{ \rm ab}(\tilde{\mathcal{A}}))^\Gamma=C^\infty(W,\mathcal{L}_\mathcal{A}\otimes_\mathcal{A}\Omega_1^{ \rm ab}(\tilde{\mathcal{A}})),\\
&\mbox{as} \quad \mathrm{d}_{\mathcal{L},h}f:=\sum_{\gamma, \,\gamma_1\gamma_2=e}f_\gamma \lambda_\gamma \gamma_1^*(h)\lambda_{\gamma_1}\mathrm{d}\gamma_2.
\end{align*}
We denote the associated total connection by $\tilde{\nabla}_{\mathcal{L},h}$. 
\end{ex}

Returning to the general case, if $\tilde{\nabla}_\mathcal{E}$ is a total connection on an $\mathcal{A}$-bundle $\mathcal{E}_\mathcal{A}\to W$, we can using the Leibniz rule extend $\tilde{\nabla}_\mathcal{E}$ to an operator that we, by an abuse of notation, denote 
\begin{align}
\nonumber
\tilde{\nabla}_\mathcal{E}:&C^\infty(W,\wedge ^*T^*W\otimes \mathcal{E}_{\mathcal{A}}\otimes_\mathcal{A}\hat{\Omega}_*^{\rm ab}(\tilde{\mathcal{A}}))\to C^\infty(W,\wedge ^*T^*W\otimes \mathcal{E}_{\mathcal{A}}\otimes_\mathcal{A}\hat{\Omega}_*^{\rm ab}(\tilde{\mathcal{A}}))\\
\label{extendingtotal}
&\mbox{such that}\quad \tilde{\nabla}_\mathcal{E}(fa)=\nabla_\mathcal{E}(f)a+(-1)^{k+m}f(\mathrm{d}_W+\mathrm{d}_\mathcal{A})a,
\end{align} 
for $a\in C^\infty(W,\tilde{\mathcal{A}})$ and $f\in C^\infty(W,\wedge ^kT^*W\otimes E_{\mathcal{A}}\otimes_\mathcal{A}\hat{\Omega}_m^{\rm ab}(\tilde{\mathcal{A}}))$. 

\begin{ex}
\label{computingcurv}
Consider the total connection $\tilde{\nabla}_{\mathcal{L},h}$ from Example \ref{coveringexample}. One computes that 
$$\tilde{\nabla}_{\mathcal{L},h}^2=\sum_{\gamma_1\gamma_2=e}\gamma_1^*(\mathrm{d}h)\gamma_1\mathrm{d}\gamma_2+\sum_{\gamma_1\gamma_2\gamma_3=e} \gamma_1^*(h)(\gamma_1\gamma_2)^*(h)\gamma_1\mathrm{d}\gamma_2\mathrm{d}\gamma_3 .$$
This can be considered an element of total degree $2$ in $C^\infty(W,\wedge ^*T^*W\otimes \Omega_*^{\rm ab}(\tilde{\mathcal{A}}))$ because $\mathrm{End}(\mathcal{L}_{\mathcal{A}})$ is trivializable. 
\end{ex}

Following \cite[Section 4.2]{DGIII} and \cite[Section 1.2]{WahlAPSForCstarBun}, we define the Chern form of a total connection on an $\mathcal{A}$-bundle $\mathcal{E}_\mathcal{A}\to W$ on a Riemannian manifold with boundary $W$ as the chain
$$\mathrm{Ch}_\mathcal{A}(\tilde{\nabla}_\mathcal{E}):=\int_W\mathrm{tr}(\mathrm{e}^{-\tilde{\nabla}_\mathcal{E}^2})\wedge \mathrm{Td}(W)\in \hat{\Omega}_*^{\rm ab}(\tilde{\mathcal{A}}).$$
Note that if $W=pt$, $\mathrm{Ch}_\mathcal{A}(\tilde{\nabla}_\mathcal{E})$ coincides with the definition of the Chern character into de Rham homology from \cite[Chapter I]{Kaast}. The following proposition follows immediately from Stokes Theorem and the fact that $\mathrm{e}^{-\tilde{\nabla}_\mathcal{E}^2}$ is $\mathrm{d}_W+\mathrm{d}_\mathcal{A}$-closed.

\begin{prop}
\label{stokesanddzero}
Assume that $\tilde{\nabla}_\mathcal{E}$ is a total connection for $(W, \mathcal{E}_{\mathcal{A}})$. Then 
$$\mathrm{d}_{\mathcal{A}}\mathrm{Ch}_{\mathcal{A}}(\tilde{\nabla}_\mathcal{E})=\mathrm{Ch}_{\mathcal{A}}(\tilde{\nabla}_{\mathcal{E}|_{\partial W}}).$$
Here the right hand side is the Chern form of the restricted total connection $\tilde{\nabla}_{\mathcal{E}|_{\partial W}}$ on the $\mathcal{A}$-bundle $\mathcal{E}_\mathcal{A}|_{\partial W}\to \partial W$.
\end{prop}

\begin{theorem}
\label{theabschen}
The Chern character 
$$\ch_\mathcal{A}:K_*^{\rm geo}(pt;\mathcal{A})\to \hat{H}^{\rm dR}_*(\mathcal{A}),\quad (M,\mathcal{E}_\mathcal{A})\mapsto [\mathrm{Ch}_\mathcal{A}(\tilde{\nabla}_\mathcal{E})],$$
for some choice of total connection, is a well defined mapping. Moreover, if $\mathcal{A}$ is holomorphically closed and dense in a $C^*$-algebra $A$ then 
\begin{equation}
\label{indexthmforch}
\ch_\mathcal{A}\left(\mathrm{ind}_{\rm AS}(M,\mathcal{E}_\mathcal{A}\otimes A)\right)=\ch_\mathcal{A}(M,\mathcal{E}_\mathcal{A}).
\end{equation}
\end{theorem}

The notation $\mathrm{ind}_{\rm AS}$ is discussed in Remark \ref{indasdef} (on page \pageref{indasdef}). Furthermore, the equality in \eqref{indexthmforch} uses the isomorphism $K_*(A)\cong K_*(\mathcal{A})$ and the definition of the Chern character into de Rham homology \cite[Chapter I]{Kaast}. 

\begin{proof}
The proof that $\ch_\mathcal{A}$ is well defined follows from standard techniques in Chern-Weil theory and Proposition \ref{stokesanddzero}. Let us prove the identity \eqref{indexthmforch}. In $K_*^{\rm geo}(pt,A)$ it holds that $(M,\mathcal{E}_\mathcal{A}\otimes A)\sim (pt, \mathrm{ind}_{\rm AS}(M,\mathcal{E}_\mathcal{A}\otimes A))$, see \cite[Proof of Lemma 6.9]{DG}. Using that $\mathcal{A}\subseteq A$ is dense and holomorphically closed, $(M,\mathcal{E}_\mathcal{A})$ is equivalent in $K_*^{\rm geo}(pt;\mathcal{A})$ to cycle of the form $(pt, x)$, where $x$ is the image of $\mathrm{ind}_{\rm AS}(M,\mathcal{E}_\mathcal{A}\otimes A)$ under the isomorphism $K_*(A)\cong K_*(\mathcal{A})$ which proves \eqref{indexthmforch}.
\end{proof}

\section{Chern characters on relative $K$-homology}
\label{subsectiononcherelhom}

We now consider Chern characters on relative $K$-homology. All homology and cohomology groups are tacitly assumed to have complex coefficients. For a continuous map $h:Y\to X$ between compact spaces, we will construct $\ch_*:K_*^{\mathrm{geo}}(h)\to H_*(h)$ where $H_*(h)$ is the relative homology of $h$. We identify the case $Y=\emptyset$ with the absolute case. The restriction to the absolute case recovers the construction of \cite{bcchern} and \cite[Section 11]{BD}. 

We work with singular homology. Let $\mathcal{C}_*^{\mathrm{sing}}(X)$ denote the complex singular chain complex of a topological space $X$ and $\mathcal{C}^*_{\mathrm{sing}}(X)$ the complex singular cochain complex. The differential in the singular complex will be denoted by $\partial$. If $X$ is a topological space, we let $H_*^{\mathrm{sing}}(X)$ denote its singular homology. We consider all homology groups as $\field{Z}/2\field{Z}$-graded groups. We recall the following standard fact for the singular chain complex on a manifold.

\begin{prop}
\label{capproductswithsingular}
Let $W$ be a manifold (possibly with boundary) and $\omega$ a closed differential form on $W$. For any singular chain $\sigma$, the cap product $\omega\cap \sigma$ is a well defined singular chain and 
$$\partial(\omega\cap \sigma)=(-1)^{|\omega|} \omega\cap \partial \sigma.$$
\end{prop}

\subsection{The absolute situation} 

Before defining the relative Chern character on $K$-homology, we recall the absolute situation from \cite[Section 11]{BD}. To simplify the relative construction, we construct the absolute Chern character $\ch_X:K_*^{\rm geo}(X)\to H_*^{\mathrm{sing}}(X)$ from cycles (equipped with additional data) to cycles. The reader is encouraged to recall the definition of a cycle with connection from Definition \ref{cycleswitcondef}. 

\begin{define}
For a cycle with connection $(M,E,f,\nabla_E)$ for $K_*^{\rm geo}(X)$, we define its Chern form as the cycle 
$$\mathrm{Ch}(M,E,f,\nabla_E):=f_*\left(\left(\mathrm{Ch}(\nabla_E) \cup \mathrm{Td}(M)\right)\cap \sigma_M\right)\in \mathcal{C}_*^{\mathrm{sing}}(X),$$
where $\mathrm{Td}(M)$ denotes the Todd form of $M$, $\mathrm{Ch}(\nabla_E)$ denotes the Chern form of $\nabla_E$ and $\sigma_M\in \mathcal{C}_*^{\mathrm{sing}}(M)$ represents the fundamental class of $M$.
\end{define}

\begin{prop}[{\cite[Page 141-142]{BD}}]
The Chern character 
$$\ch_X:K_*^{\rm geo}(X)\to H_*^{\mathrm{sing}}(X), \quad (M,E,f)\mapsto [\mathrm{Ch}(M,E,f,\nabla_E)],$$
for some choice of connection, is well defined. Moreover, if $X$ is a locally finite $CW$-complex, $\ch_X$ is a complex isomorphism. 
\end{prop}

\subsection{The relative situation}
\label{subseconrelhomonkjom}

For a continuous mapping $h:Y\to X$, we let $\mathcal{C}_*^{\mathrm{rel}}(h)$ denote the mapping cone complex associated with $h_*:\mathcal{C}_*^{\mathrm{sing}}(Y)\to \mathcal{C}_*^{\mathrm{rel}}(X)$. That is, $\mathcal{C}_*^{\mathrm{rel}}(h):=\mathcal{C}_{*}^{\mathrm{sing}}(X)\oplus \mathcal{C}_{*-1}^{\mathrm{sing}}(Y)$ equipped with the differential 
$$\partial^{\mathrm{rel}}:=\begin{pmatrix} \partial_X&-h_*\\ 0& \partial_Y\end{pmatrix}.$$
The association $h\mapsto \mathcal{C}_*^{\mathrm{rel}}(h)$ is functorial for mappings $(f,g):[h:Y_1\to X_1]\to [h_2:Y_2\to X_2]$. The induced mapping $(f,g)_*:\mathcal{C}_*^{\mathrm{rel}}(h_1)\to \mathcal{C}_*^{\mathrm{rel}}(h_2)$ is defined as the restriction of $f_*\oplus g_*:\mathcal{C}_{*}^{\mathrm{sing}}(X_1)\oplus \mathcal{C}_{*-1}^{\mathrm{sing}}(Y_1)\to \mathcal{C}_{*}^{\mathrm{sing}}( X_2)\oplus \mathcal{C}_{*-1}^{\mathrm{sing}}(Y_2)$.

Recall that for a manifold with boundary, the fundamental class in homology is a relative class $[W]\in H_*(W,\partial W)$ represented by a relative cycle $(\sigma_W,\sigma_{\partial W})\in \mathcal{C}_*^{\mathrm{rel}}(i)$, where $i$ denotes the inclusion $i:\partial W\hookrightarrow W$.

\begin{remark}
\label{boundaryrem}
A special and important case of Proposition \ref{capproductswithsingular} is the following. Let $W$ be a Riemannian manifold collared near its boundary. Represent the fundamental class in homology by a relative cycle $(\sigma_W,\sigma_{\partial W})\in \mathcal{C}_*^{\mathrm{rel}}(i)$. If $\nabla_E$ is a connection on a vector bundle $E\to W$, we can from Proposition \ref{capproductswithsingular} deduce that $(\mathrm{Ch}(\nabla_E)\wedge \mathrm{Td}(W))\cap \sigma_W$ is a well defined singular chain such that 
$$\partial ((\mathrm{Ch}(\nabla_E)\wedge \mathrm{Td}(W))\cap \sigma_W)=i_*((\mathrm{Ch}(\nabla_E|_{\partial W})\wedge \mathrm{Td}(\partial W))\cap \sigma_{\partial W}).$$ 
In particular, if the boundary of $W$ is empty, $(\mathrm{Ch}(\nabla_E)\wedge \mathrm{Td}(W))\cap \sigma_W$ is closed.
\end{remark}

We can by the same argument as in the proof of Proposition \ref{capproductswithsingular} deduce the next result.

\begin{prop}
\label{capprodwithforms}
For a manifold with boundary $W$, we let $\Omega^*_{\mathrm{dR}}(W)$ denote the de Rham complex of differential forms and $i:\partial W\to W$ the boundary inclusion. The cap product gives rise to a morphism of complexes
$$\Omega^{-*}_{\mathrm{dR}}(W)\otimes \mathcal{C}_*^{\mathrm{rel}}(i)\to \mathcal{C}_*^{\mathrm{rel}}(i), \quad \omega\otimes \begin{pmatrix} \sigma_1\\\sigma_2\end{pmatrix}\mapsto \begin{pmatrix} \omega\cap \sigma_1\\\omega|_{\partial W}\cap\sigma_2\end{pmatrix}.$$
\end{prop}

\begin{notation}
If $h:Y\to X$ is a continuous mapping of topological spaces we let $H_*^{\mathrm{rel}}(h)$ denote the homology of $\mathcal{C}_{*}^{\mathrm{rel}}(h)$. The vector space $H_*^{\mathrm{rel}}(h)$ can be considered a module for $H^*_{\mathrm{sing}}(X)$. We remind the reader that we consider all homology groups as $\field{Z}/2\field{Z}$-graded groups.
\end{notation}

\begin{prop}
\label{sesforsinghom}
The association $h\mapsto H_*^{\mathrm{rel}}(h)$ is functorial in the sense that if $(f,g):[h_1:Y_1\to X_1]\to [h_2:Y_2\to X_2]$, there is an induced mapping $(f,g)_*:H_*^{\mathrm{rel}}(h_1)\to H_*^{\mathrm{rel}}(h_2)$. Moreover, for a continuous mapping $h:Y\to X$, the short exact sequence of complexes
$$0\to \mathcal{C}_{*}^{\mathrm{sing}}(X)\to \mathcal{C}_*^{\mathrm{rel}}(h)\to \mathcal{C}_{*-1}^{\mathrm{sing}}(Y)\to 0,$$
gives rise to a long exact sequence in homology
$$\cdots \to H_*^{\mathrm{sing}}(Y)\xrightarrow{h_*} H_*^{\mathrm{sing}}(X)\xrightarrow{r} H_*^{\mathrm{rel}}(h)\xrightarrow{\delta} H_{*-1}^{\mathrm{sing}}(Y)\to \cdots,$$
where $r:H_*^{\mathrm{sing}}(X)\to H_*^{\mathrm{rel}}(h)$ is defined from the inclusion of complexes $\mathcal{C}_{*}^{\mathrm{sing}}(X)\hookrightarrow \mathcal{C}_*^{\mathrm{rel}}(h)$ and $\delta:H_*^{\mathrm{rel}}(h)\to H_{*-1}^{\mathrm{sing}}(Y)$ denotes the boundary mapping.
\end{prop}

The proposition follows from the standard construction of long exact sequences in homology. We are now ready to define the Chern character on relative $K$-homology. We say that $(W,E,(f,g),\nabla_E)$ is a cycle with connection for $K_*^{\mathrm{geo}}(h)$ if $(W,E,(f,g),\nabla_E)$ is a cycle for $K_*^{\mathrm{geo}}(h)$ and $\nabla_E$ is a Hermitean connection on $E$ that is of product type near $\partial W$. 

\begin{define}
Let $h:X\to Y$ be a continuous mapping of compact spaces. Let $(W,E,(f,g),\nabla_E)$ be a cycle with connection for $K_*^{\mathrm{geo}}(h)$. For a representative $(\sigma_W,\sigma_{\partial W})\in \mathcal{C}_*^{\mathrm{rel}}(i)$ of the fundamental class in homology of $W$ we define 
$$\mathrm{Ch}(W,E,(f,g),\nabla_E):=
(f,g)_*\begin{pmatrix} 
(\ch(\nabla_E)\wedge \mathrm{Td}(W))\cap \sigma_W\\
(\ch(\nabla_E|_{\partial W})\wedge \mathrm{Td}(\partial W))\cap \sigma_{\partial W}
\end{pmatrix}\in \mathcal{C}_*^{\mathrm{rel}}(h).$$
\end{define}

As a chain, $\mathrm{Ch}(W,E,(f,g),\nabla_E)$ also depends on $(\sigma_W,\sigma_{\partial W})$. We will suppress this dependence in the notation because its choice will by Proposition \ref{capprodwithforms} not affect the homology class.

\begin{prop}
The chain $\mathrm{Ch}(W,E,(f,g),\nabla_E)\in \mathcal{C}_*^{\mathrm{rel}}(h)$ is closed and the associated class 
\begin{equation}
\label{cherndefine}
\ch_{\rm rel}^h(W,E,(f,g)):=[\mathrm{Ch}(W,E,(f,g),\nabla_E)]\in H_*^{\mathrm{rel}}(h),
\end{equation}
does not depend on the choice of connection nor the choice of $(\sigma_W,\sigma_{\partial W})$.
\end{prop}

\begin{proof}
It follows from Remark \ref{boundaryrem} that $\mathrm{Ch}(W,E,(f,g),\nabla_E)$ is closed. The class $[\mathrm{Ch}(W,E,(f,g),\nabla_E)]$ coincides with the cap product of $[\ch(\nabla_E)\wedge \mathrm{Td}(W)]\in H^*_{\mathrm{dR}}(W)$ with the fundamental class $[W]:=[(\sigma_W,\sigma_{\partial W})]\in H_*^{\mathrm{rel}}(W,\partial W)$ and is therefore independent of choices.
\end{proof}

\begin{theorem}
\label{cheronhom}
The Chern character from Equation \eqref{cherndefine} induces a well defined natural mapping $\ch_{\mathrm{rel}}^h:K_*^{\mathrm{geo}}(h)\to H_*^{\mathrm{rel}}(h)$. The Chern character fits into a commuting diagram with exact rows:
\begin{center}
$\begin{CD}
@>>> K_*^{\rm geo}(Y) @>h_*>> K_*^{\rm geo}(X) @>r>> K_*^{\rm geo}(h) @>\delta>> K_{*-1}^{\rm geo}(Y) @>>> \\
@. @VV\ch_YV @VV\ch_XV  @VV\ch_{\rm rel}^h V  @VV\ch_Y V\\
@>>> H_*^{\rm sing}(Y) @>h_*>>H_*^{\rm sing}(X) @>r>> H_*^{\rm rel}(h) @>\delta>> H_{*-1}^{\rm sing}(Y) @>>> \\
\end{CD}$
\end{center}
If $(X,Y)$ is a locally finite CW-pair, the vertical mappings are complex isomorphisms.
\end{theorem}

\begin{proof}
By construction, the diagram commutes if the mappings are well defined. By the five lemma, all vertical mappings are complex isomorphisms for locally finite CW-pairs if they are well defined since they are complex isomorphisms in the absolute case for locally finite CW-complexes. Therefore, the proof of the theorem is complete once proving that the Chern character from Equation \eqref{cherndefine} is well defined from classes to classes, that is, the Chern character respects the disjoint union/direct sum relation, vector bundle modification and the bordism relation. The disjoint union/direct sum relation is trivially satisfied. We leave the proof that the Chern character respects vector bundle modification and bordism to the reader; the proofs follows  \cite[Section 4]{DGIII}, see \cite[Proposition 4.21]{DGIII} and \cite[Lemma 4.19]{DGIII}, respectively. See also \cite[Section 11]{BD}.
\end{proof}

\begin{remark}
Let $h:Y\to X$ be a continuous mapping and set $\Gamma_1=\pi_1(Y)$ and $\Gamma_2=\pi_1(X)$. The mapping $h$ induces a mapping $Bh:B\Gamma_1\to B\Gamma_2$. There are universal mappings $f_Y:Y\to B\Gamma_1$ and $f_X:X\to B\Gamma_2$. Under the additional assumption that $(f_X,f_Y):[h:Y\to X]\to [Bh:B\Gamma_1\to B\Gamma_2]$ we arrive at a commuting diagram with exact rows:
\small
\begin{center}
$\begin{CD}
@>>> H_*^{\rm sing}(Y) @>h_*>>H_*^{\rm sing}(X) @>r>> H_*^{\rm rel}(h) @>\delta>> H_{*-1}^{\rm sing}(Y) @>>> \\
@. @VV(f_Y)_*V @VV(f_X)_*V  @VV(f_X,f_Y)_* V  @VV(f_Y)_* V\\
@>>> H_*^{\rm sing}(B\Gamma_1) @>(Bh)_*>>H_*^{\rm sing}(B\Gamma_2) @>r>> H_*^{\rm rel}(Bh) @>\delta>> H_{*-1}^{\rm sing}(B\Gamma_1) @>>> \\
\end{CD}$
\end{center}
\normalsize
A similar sequence exists on $K$-homology and the Chern character is compatible with the two sequences.
\end{remark}

For simplicity, we often drop the ``$\mathrm{sing}$" from the superscript on homology.

\section{Chern characters on relative $K$-theory}
\label{chernonk}

The Chern character we use on relative $K$-theory uses noncommutative de Rham homology. This is done as in Section \ref{subseconder} following the ideas of \cite[Section 4]{DGIII}. The usage of noncommutative de Rham homology is compatible with several of the higher index theorems found in the literature \cite{Lott,SchL2,wahlProFor,WahlAPSForCstarBun,WahlSurSet}.

\subsection{Relative noncommutative de Rham theory} 
\label{relativedersubs}
As above, we consider a continuous $*$-homomorphism $\phi_\mathcal{A}:\mathcal{A}_1\to \mathcal{A}_2$. Its relative de Rham homology is easily defined from the absolute complexes. The reader is referred back to Section \ref{subseconder} for notations.

\begin{define}
We define $\hat{\Omega}_*^{\rm ab}(\phi_\mathcal{A})$ as the mapping cone complex of the morphism $(\tilde{\phi}_\mathcal{A})_*:\hat{\Omega}_*^{\rm ab}(\tilde{\mathcal{A}}_1)\to \hat{\Omega}_*^{\rm ab}(\tilde{\mathcal{A}}_2)$, that is $\hat{\Omega}_*^{\rm ab}(\phi_\mathcal{A}):=\hat{\Omega}_*^{\rm ab}(\tilde{\mathcal{A}}_2)\oplus \hat{\Omega}_{*+1}^{\rm ab}(\tilde{\mathcal{A}}_1)$ equipped with the differential 
$$\begin{pmatrix} \mathrm{d}_{\mathcal{A}_2}&-(\phi_\mathcal{A})_*\\ 0& \mathrm{d}_{\mathcal{A}_1}\end{pmatrix}.$$
The relative de Rham homology of $\phi_\mathcal{A}$ is defined as the topological homology of $\hat{\Omega}_*^{\rm ab}(\phi_\mathcal{A})$ and is denoted by $\hat{H}^{\rm rel}_*(\phi_{\mathcal{A}})$.
\end{define}

It follows from the construction that $\hat{\Omega}_*^{\rm ab}(\phi_\mathcal{A})$ fits into an admissible\footnote{That is, one admitting a continuous linear splitting.} short exact sequence of complexes
\begin{equation}
\label{sesforco}
0\to \hat{\Omega}_*^{\rm ab}(\tilde{\mathcal{A}}_2)\to \hat{\Omega}_*^{\rm ab}(\phi_\mathcal{A})\to\hat{\Omega}_{*+1}^{\rm ab}(\tilde{\mathcal{A}}_1)\to 0.
\end{equation}
The next proposition follows from the admissibility of this short exact sequence.

\begin{prop}
\label{yeatanotherles}
Let $\phi_\mathcal{A}:\mathcal{A}_1\to \mathcal{A}_2$ be a continuous $*$-homomorphism. The short exact sequence \eqref{sesforco} induces a long exact sequence
$$\cdots \to \hat{H}_*^{\rm dR}(\mathcal{A}_1) \xrightarrow{(\phi_\mathcal{A})_*}\hat{H}_*^{\rm dR}(\mathcal{A}_2)\xrightarrow{r}  \hat{H}_*^{{\rm rel}}(\phi_\mathcal{A}) \xrightarrow{\delta} \hat{H}_{*+1}^{\rm dR}(\mathcal{A}_1)\xrightarrow{(\phi_\mathcal{A})_*}\cdots $$
\end{prop}

The map $r:\hat{H}_*^{\rm dR}(\mathcal{A}_2)\to \hat{H}_*^{{\rm rel}}(\phi_\mathcal{A}) $ is defined from the inclusion $\hat{\Omega}_*^{\rm ab}(\tilde{\mathcal{A}}_2)\to \hat{\Omega}_*^{\rm ab}(\phi_\mathcal{A})$. The map $\delta:\hat{H}_*^{{\rm rel}}(\phi_\mathcal{A}) \to \hat{H}_{*+1}^{\rm dR}(\mathcal{A}_1)$ is defined from the projection $\hat{\Omega}_*^{\rm ab}(\phi_\mathcal{A})\to\hat{\Omega}_{*+1}^{\rm ab}(\tilde{\mathcal{A}}_1)$.

\begin{remark}
\label{remarkonrelloc}
We define $\Omega_*^{\rm <e>}(\phi)$ as the mapping cone complex of the morphism $\tilde{\phi}_*:\Omega_*^{\rm <e>}(\C[\Gamma_1])\to \Omega_*^{\rm <e>}(\C[\Gamma_2])$, that is $\Omega_*^{{\rm <e>}}(\phi):=\Omega_*^{\rm <e>}(\C[\Gamma_2])\oplus \Omega_{*+1}^{\rm <e>}(\C[\Gamma_1])$
equipped with the differential 
$$\begin{pmatrix} \mathrm{d}_{\mathcal{A}_2}&-\phi_*\\ 0& \mathrm{d}_{\mathcal{A}_1}\end{pmatrix}.$$
The complex $\Omega_*^{\rm <e>}(\phi)$ is a sub-complex of $\Omega_*^{\rm ab}(\phi)$. We define $H_*^{{\rm rel},<e>}(\phi)$ as the homology of $\Omega_*^{\rm <e>}(\phi)$. 
Proposition \ref{hstarhstar} is based on the work of Burghelea \cite{burgher} whose work implies that for a discrete group $\Gamma$ there is a quasi-isomorphism 
$$\mathcal{C}_*(\Gamma)\otimes CC_*(\C)\to \Omega_*^{<e>}(\C[\Gamma]),$$
where $\mathcal{C}_*(\Gamma)$ denotes the complex of chains for the group homology of $\Gamma$ and $CC_*(\C)$ the cyclic complex for $\C$. Let $\mathcal{C}_*(\phi)$ denote the mapping cone complex of $\phi_*:\mathcal{C}_*(\Gamma_1)\to \mathcal{C}_*(\Gamma_2)$, note that the homology of $\mathcal{C}_*(\phi)$ coincides with $H_*^{\rm rel}(B\phi)$ up to isomophism. By functoriality, we arrive at a commuting diagram of complexes with exact rows and all vertical maps being quasi-isomorphisms:
\tiny
\begin{center}
$\begin{CD}
0 @>>>\mathcal{C}_*(\Gamma_2)\otimes CC_*(\C)@>>> \mathcal{C}_*(\phi)\otimes CC_*(\C) @>>> \mathcal{C}_{*-1}(\Gamma_1)\otimes CC_*(\C)@>>> 0\\
@. @VVV  @VVV  @VVV\\
0 @>>> \Omega_*^{<e>}(\C[\Gamma_2]) @>>>  \Omega_*^{<e>}(\phi)@>>>  \Omega_{*+1}^{<e>}(\C[\Gamma_1]) @>>>0 \\
\end{CD}$
\end{center}
\normalsize
It follows that there is an isomorphism $\beta_{\rm rel}:H_*^{{\rm rel},<e>}(\phi)\xrightarrow{\sim} H_*^{\rm rel}(B\phi)\otimes HC_{*}(\C)$ fitting into a commuting diagram with exact rows and vertical mappings being isomorphisms:
\tiny
\begin{center}
$\begin{CD}
@>>> H_*^{<e>}(\C[\Gamma_1]) @>>>H_*^{<e>}(\C[\Gamma_2])@>>> H_*^{{\rm rel}, <e>}(\phi) @>>> H_{*+1}^{<e>}(\C[\Gamma_1])@>>> \\
@. @VV\beta_{\Gamma_1}V @VV\beta_{\Gamma_2}V  @VV\beta_{\rm rel} V  @VV\beta_{\Gamma_1} V\\
@>>> H_*^{\rm sing}(B\Gamma_1)\otimes R_* @>>>H_*^{\rm sing}(B\Gamma_2)\otimes R_* @>>> H_*^{\rm rel}(B\phi)\otimes R_*@>>> H_{*-1}^{\rm sing}(B\Gamma_1)\otimes R_* @>>> \\
\end{CD}$
\end{center}
\normalsize
where $R_*:=HC_{*}(\C)$. The top row is the long exact sequence comes from the localized version of \eqref{sesforco} and the bottom row is can be found in Proposition \ref{sesforsinghom}. We can thus define a first version of relative assembly at the level of homology. Note that $H_*^{{\rm rel}, <e>}(\phi)\subseteq H_*^{{\rm rel}}(\phi)$ is a direct summand (the analogous statement need not be true for $\hat{H}_*^{{\rm rel}}(\phi_\mathcal{A})$).
\end{remark}

\begin{define}
\label{inclusionassembly}
Let $\phi:\Gamma_1\to \Gamma_2$ be a group homomorphism extending continuously to a $*$-homomorphism of Frechet $*$-algebra completions $\phi_\mathcal{A}:\mathcal{A}_1\to \mathcal{A}_2$ of $\C[\Gamma_1]$ and $\C[\Gamma_2]$, respectively. Define the map $\mu_{\rm dR,\phi}$ to make the following diagram commute:
$$\xymatrix{
 H_*^{\rm rel}(B\phi)=H_*^{\rm rel}(B\phi)\otimes HC_{0}(\C)\ar[r]^{\mbox{                }\qquad\qquad\quad\mu_{\rm dR,\phi}}\ar[d]
 & \hat{H}_*^{{\rm rel}}(\phi_\mathcal{A})  
 \\ 
H_*^{\rm rel}(B\phi)\otimes HC_{*}(\C)\ar[r]_{\beta_{\mathrm{rel}}^{-1}}
& H_*^{{\rm rel}, <e>}(\phi)\ar[u]   
}$$
The mapping $H_*^{{\rm rel}, <e>}(\phi)\to \hat{H}_*^{{\rm rel}}(\phi_\mathcal{A})$ is defined from the inclusion of complexes $\Omega_*^{\rm <e>}(\phi)\hookrightarrow \hat{\Omega}_*^{\rm ab}(\phi_\mathcal{A})$.
\end{define}

\subsection{Chern characters on $K_*^{\rm geo}(pt; \phi_\mathcal{A})$}
\label{subsectiononcherelkthe}

The Chern character on relative $K$-theory is defined similarly as in Section \ref{subseconder} using Chern-Weil theory. See also \cite[Section 4]{DGIII} and \cite[Section 1.2]{WahlAPSForCstarBun}.

\begin{define}
A collection $(W, (\mathcal{E}_{\mathcal{A}_2}, \mathcal{F}_{\mathcal{A}_1}, \alpha),(\tilde{\nabla}_\mathcal{E},\tilde{\nabla}_\mathcal{F}))$ is called a cycle with total connections for $K_*^{\rm geo}(pt; \phi_\mathcal{A})$ whenever $(W, (\mathcal{E}_{\mathcal{A}_2}, \mathcal{F}_{\mathcal{A}_1}, \alpha))$ is a cycle for $K_*^{\rm geo}(pt; \phi_\mathcal{A})$, $\tilde{\nabla}_\mathcal{E}$ is a total connection for $\mathcal{E}_{\mathcal{A}_2}$ and $\tilde{\nabla}_\mathcal{F}$ is a total connection for $\mathcal{F}_{\mathcal{A}_1}$ satisfying that near $\partial W$, 
\begin{equation}
\label{boundarycompatibility}
\tilde{\nabla}_\mathcal{E}=\alpha^*(\tilde{\nabla}_\mathcal{F}\otimes_{\phi_\mathcal{A}}1_{\mathcal{A}_2}).
\end{equation}
In this case, we say that $(\tilde{\nabla}_\mathcal{E},\tilde{\nabla}_\mathcal{F})$ is a total connection for $(W, (\mathcal{E}_{\mathcal{A}_2}, \mathcal{F}_{\mathcal{A}_1}, \alpha))$.
\end{define}

\begin{remark}
The notion of connection in the context of geometric cycles for surgery was defined in \cite[Definition 4.10]{DGIII}. Connections in that context are more complicated due to the existence of ``non-easy cycles" (Such cycles are needed in the surgery case because the Mishchenko bundle on the boundary need not extend, see \cite[Introduction]{DGII} for more on this). The distinction between connections and total connections was not emphasized in \cite{DGIII}.
\end{remark}

\begin{ex}
Any cycle $(W, (\mathcal{E}_{\mathcal{A}_2}, \mathcal{F}_{\mathcal{A}_1}, \alpha))$ for $K_*^{\rm geo}(pt; \phi_\mathcal{A})$ admits a connection. Both the bundles $\mathcal{E}_{\mathcal{A}_2}$ and $\mathcal{F}_{\mathcal{A}_1}$ admit total Gra\ss mannian connections $\tilde{\nabla}_\mathcal{E}^{\rm gr}$ and $\tilde{\nabla}_\mathcal{F}^{\rm gr}$ that lift the Gra\ss mannian connections $\nabla_\mathcal{E}^{\rm gr}$ and $\nabla_\mathcal{F}^{\rm gr}$, respectively. Take a cutoff function $\chi$ supported in a collar neighborhood of $\partial W$ with $\chi=1$ near $\partial W$. The pair $\left((1-\chi)\tilde{\nabla}_\mathcal{E}^{\rm gr}+\chi\alpha^*(\tilde{\nabla}_\mathcal{F}^{\rm gr}\otimes_{\phi_\mathcal{A}} 1_{\mathcal{A}_2}),\tilde{\nabla}_\mathcal{F}^{\rm gr}\right)$ will form a total connection for the cycle $(W, (\mathcal{E}_{\mathcal{A}_2}, \mathcal{F}_{\mathcal{A}_1}, \alpha))$.
\end{ex}

\begin{ex}
\label{coveringexamplerelative}
For assembled cycles from $K_*^{\rm geo}(B\phi)$, for a group homomorphism $\phi:\Gamma_1\to \Gamma_2$, there is particular choice of connections following Example \ref{coveringexample}. Consider Fr\'echet $*$-algebra completions $\mathcal{A}_1$ and $\mathcal{A}_2$ of group algebras $\C[\Gamma_1]$ and $\C[\Gamma_2]$, respectively, such that $\phi$ extends to a continuous $*$-homomorphism $\mathcal{A}_1\to \mathcal{A}_2$. From $f$ and $g$, we define the $\Gamma_2$-Galois covering $\tilde{W}:=E\Gamma_2\times_fW\to W$ and the $\Gamma_1$-Galois covering $\widetilde{\partial W}:=E\Gamma_1\times_g \partial W\to \partial W$. We can pick functions $h\in C^\infty_c(\tilde{W},[0,1])$ and $h_\partial \in C^\infty_c(\widetilde{\partial W},[0,1])$ as above arriving at total connections $\tilde{\nabla}_{\mathcal{L},h}$ and $\tilde{\nabla}_{\mathcal{L},h_\partial}$ for $f^*\mathcal{L}_{\mathcal{A}_2}\to W$ and $g^*\mathcal{L}_{\mathcal{A}_1}\to \partial W$, respectively, as in Example \ref{coveringexample}. 

If $\nabla_E$ is a hermitean connection on $E$, the total connections $\nabla_E\otimes \tilde{\nabla}_{\mathcal{L},h}$ and $\nabla_{E|_{\partial W}}\otimes\tilde{\nabla}_{\mathcal{L},h_\partial} $ lift the connections $\nabla_E$ on $E\otimes f^*\mathcal{L}_{\mathcal{A}_2}\to W$ and $\nabla_{E|_{\partial W}}$ on $E|_{\partial W}\otimes g^*\mathcal{L}_{\mathcal{A}_1}\to \partial W$, respectively. Take a cutoff function $\chi$ supported in a collar neighborhood of $\partial W$ with $\chi=1$ near $\partial W$. The pair 
$$\left((1-\chi)\nabla_E\otimes \tilde{\nabla}_{\mathcal{L},h}+\chi\alpha^*(\nabla_{E|_{\partial W}}\otimes\tilde{\nabla}_{\mathcal{L},h_\partial} \otimes_{\phi_\mathcal{A}} 1_{\mathcal{A}_2}),\nabla_{E|_{\partial W}}\otimes\tilde{\nabla}_{\mathcal{L},h_\partial} \right)$$ 
will form a total connection for the cycle $\mu_\phi(W,E,(f,g))$. Here $\alpha$ denotes the isomorphism from Lemma \ref{isomopfmish}.
\end{ex}

Returning to the general case, if $(\tilde{\nabla}_\mathcal{E},\tilde{\nabla}_\mathcal{F})$ is a total connection for $(W, (\mathcal{E}_{\mathcal{A}_2}, \mathcal{F}_{\mathcal{A}_1}, \alpha))$ we can extend the total connections as in Equation \eqref{extendingtotal}. It follows from Proposition \ref{stokesanddzero} that 
\begin{equation}
\label{stokesandd}
\mathrm{d}_{\mathcal{A}_2}\mathrm{Ch}_{\mathcal{A}_2}(\tilde{\nabla}_\mathcal{E})=(\phi_\mathcal{A})_*\mathrm{Ch}_{\mathcal{A}_1}(\tilde{\nabla}_\mathcal{F}).
\end{equation}
It follows from Equation \eqref{stokesandd} that the chain 
$$\begin{pmatrix}\mathrm{Ch}_{\mathcal{A}_2}(\tilde{\nabla}_\mathcal{E})\\\mathrm{Ch}_{\mathcal{A}_1}(\tilde{\nabla}_\mathcal{F})\end{pmatrix}\in \hat{\Omega}_*^{\rm ab}(\phi_\mathcal{A}),$$ 
is a cycle. 

\begin{prop}
\label{cyclestoclasses}
The Chern character 
\begin{align*}
(W, (\mathcal{E}_{\mathcal{A}_2}, &\mathcal{F}_{\mathcal{A}_1}, \alpha))\\
&\mapsto \ch_{\rm rel}^{\phi_\mathcal{A}}(W, (\mathcal{E}_{\mathcal{A}_2}, \mathcal{F}_{\mathcal{A}_1}, \alpha)):=\left[\begin{pmatrix}\mathrm{Ch}_{\mathcal{A}_2}(\tilde{\nabla}_\mathcal{E})\\\mathrm{Ch}_{\mathcal{A}_1}(\tilde{\nabla}_\mathcal{F})\end{pmatrix}\right]\in \hat{H}_*^{\rm dR}(\phi_\mathcal{A}),
\end{align*}
is well defined from cycles to classes. That is, the homology class is independent of the choice of total connection.
\end{prop}

The proposition follows from the ordinary considerations of relative Chern-Weil theory, for the detailed argument see the proof of \cite[Lemma 4.15]{DGIII}. 

\begin{remark}
The compatibility condition on the total connections near the boundary (see Equation \eqref{boundarycompatibility}) can be removed at the cost of transgression terms appearing in Equation \eqref{stokesandd} and the definition of the Chern character in Proposition \ref{cyclestoclasses} (cf. \cite[Definition 4.14]{DGIII}). We remark that the analogue of the compatibility condition in Equation \eqref{boundarycompatibility} does not help when defining the delocalized Chern character in \cite[Section 4]{DGIII}; in the present case, it simplifies the computations substantially.
\end{remark}

\begin{remark}
\label{comattcuelc}
The relative Chern character is related to the absolute Chern characters via the maps $r$ and $\delta$ from Theorem \ref{lesforphi} (see page \pageref{lesforphi}) and Proposition \ref{yeatanotherles} (see page \pageref{yeatanotherles}). Indeed, if $(M,\mathcal{E}_{\mathcal{A}_2})$ is a cycle for $K_*^{\rm geo}(pt;\mathcal{A}_2)$, then 
$$\ch_{\rm rel}^{\phi_\mathcal{A}}\circ r(M,\mathcal{E}_{\mathcal{A}_2})=\left[\begin{pmatrix}\mathrm{Ch}_{\mathcal{A}_2}(\tilde{\nabla}_\mathcal{E})\\0\end{pmatrix}\right]=r\circ \ch_{\mathcal{A}_2}(M,\mathcal{E}_{\mathcal{A}_2}),$$
because $M$ has empty boundary. Moreover, if $(W, (\mathcal{E}_{\mathcal{A}_2}, \mathcal{F}_{\mathcal{A}_1}, \alpha))$ is a cycle for $K_*^{\rm geo}(pt;\phi_\mathcal{A})$, then 
\begin{align*}
\delta\circ \ch_{\rm rel}^{\phi_\mathcal{A}}(W, (\mathcal{E}_{\mathcal{A}_2}, \mathcal{F}_{\mathcal{A}_1}, \alpha))&=\delta\left[\begin{pmatrix}\mathrm{Ch}_{\mathcal{A}_2}(\tilde{\nabla}_\mathcal{E})\\\mathrm{Ch}_{\mathcal{A}_1}(\tilde{\nabla}_\mathcal{F})\end{pmatrix}\right]=\\
&=[\mathrm{Ch}_{\mathcal{A}_1}(\tilde{\nabla}_\mathcal{F})]=\ch_{\mathcal{A}_1}\circ \delta(W, (\mathcal{E}_{\mathcal{A}_2}, \mathcal{F}_{\mathcal{A}_1}, \alpha)).
\end{align*}
\end{remark}

\begin{theorem}
\label{cheronkthe}
The Chern character defined in Proposition \ref{cyclestoclasses} induces a well defined mapping $\ch_{\rm rel}:K_*^{\rm geo}(pt;\phi_\mathcal{A})\to \hat{H}_*^{\rm rel}(\phi_\mathcal{A})$ that fits into a commuting diagram with exact rows:
\scriptsize
\begin{equation}
\label{righarla}
\begin{CD}
@>>> K_*^{\rm geo}(pt; \mathcal{A}_1) @>(\phi_\mathcal{A})_*>> K_*^{\rm geo}(pt;\mathcal{A}_2) @>r>> K_*^{\rm geo}(pt;\phi_\mathcal{A}) @>\delta>> K_{*-1}^{\rm geo}(pt;\mathcal{A}_1) @>>> \\
@. @VV\ch_{\mathcal{A}_1}V @VV\ch_{\mathcal{A}_2}V  @VV\ch_{\rm rel}^{\phi_\mathcal{A}} V  @VV\ch_{\mathcal{A}_1} V\\
@>>> \hat{H}_*^{\rm dR}(\mathcal{A}_1) @>(\phi_\mathcal{A})_*>>\hat{H}_*^{\rm dR}(\mathcal{A}_1) @>>> \hat{H}_*^{\rm rel}(\phi_\mathcal{A}) @>>> \hat{H}_{*+1}^{\rm dR}(\mathcal{A}_1)@>>> \\
\end{CD}
\end{equation}
\normalsize
The top row is the long exact sequence appearing in the diagram \eqref{yeatanothercommuting} and the bottom row is from Proposition \ref{yeatanotherles}. The absolute Chern characters $\ch_{\mathcal{A}_i}:K_*^{\rm geo}(pt; \mathcal{A}_i)\to \hat{H}_*^{\rm dR}(\mathcal{A}_i)$ are defined as above in Theorem \ref{theabschen}.
\end{theorem}

The reader should note that the different gradings in the right most part of the diagram \eqref{righarla} makes no difference since we are working with $\Z/2$-graded homology theories. The discrepancy can informally be explained as coming from the fact that in the top row the dimension of cycles is counted positively, while the integration appearing in the Chern character subtracts the dimension from the total degree of the Chern forms in $C^\infty(W,\wedge^* T^*W\otimes \hat{\Omega}^{\rm ab}(\mathcal{A}))$ and dimensions are therefore counted negatively in the bottom row. 

The proof of Theorem \ref{cheronkthe} is similar to the analogous result in \cite{DGIII}. An outline is as follows. The diagram commutes at the level of cycles by Remark \ref{comattcuelc}. In particular, the theorem follows from the five lemma as long as the maps are well defined. The absolute Chern characters are well defined by Theorem \ref{theabschen} (see page \pageref{theabschen}). The theorem follows once the Chern character from Proposition \ref{cyclestoclasses} is well defined from classes to classes. That is, the Chern character respects the disjoint union/direct sum relation, vector bundle modification and the bordism relation. The disjoint union/direct sum relation is trivially satisfied. The proof that the Chern character respects bordism follows as in \cite[Lemma 4.19]{DGIII} and the proof that it respects vector bundle modification follows as in \cite[Proposition 4.21]{DGIII}.

\begin{ex}
\label{mishcurv}
We return to the connections constructed in Example \ref{coveringexamplerelative}. We fix a cycle $(W,W\times \C, (f,g))$ with trivial bundle data, a Riemannian structure on $W$ and functions $h$ and $h_\partial$ as in Example \ref{coveringexamplerelative}. One computes $\tilde{\nabla}_{\mathcal{L},h}^2$ and $\tilde{\nabla}_{\mathcal{L},h_\partial}^2$ as in Example \ref{computingcurv}. To shorten notation, set $\bar{\nabla}_{W,f,g}:=(1-\chi)\tilde{\nabla}_{\mathcal{L},h}+\chi\alpha^*(\tilde{\nabla}_{\mathcal{L},h_\partial} \otimes_{\phi_\mathcal{A}} 1_{\mathcal{A}_2})$ as a total connection on $f^*\mathcal{L}_{\mathcal{A}_2}$. We define 
\begin{align*}
\omega_{(W,(f,g))}&:= \mathrm{e}^{-\bar{\nabla}_{W,f,g}^2}\wedge \mathrm{Td}(W)\in C^\infty(W,\wedge ^*T^*W\otimes\Omega_*^{\rm ab}(\mathcal{A}_2)),\quad\mbox{and}\\
\omega_{(\partial W,g)}&:=\mathrm{e}^{-\tilde{\nabla}_{\mathcal{L},h_\partial}^2}\wedge \mathrm{Td}(\partial W)\in C^\infty(\partial W,\wedge ^*T^*\partial W\otimes\Omega_*^{\rm ab}(\mathcal{A}_1)).
\end{align*}
It follows from the construction that $(\int_W\omega_{(W,(f,g))},\int_{\partial W}\omega_{(\partial W,g)})^T\in \hat{\Omega}_*^{\rm rel}(\phi_{\mathcal{A}})$ represents $\ch_{\rm rel}^{\phi_\mathcal{A}}(W,W\times \C,(f,g))$. Note that $(\int_W\nu\wedge\omega_{(W,(f,g))},\int_{\partial W}\nu\wedge\omega_{(\partial W,g)} )^T\in \Omega_*^{{\rm rel},<e>}(\phi)$ (for notation, see Remark \ref{remarkonrelloc}) for any differential form $\nu$ on $W$. 

\end{ex}

\section{Assembly in homology and index pairings}
\label{assinhom}

The free assembly map has an analogue in homology. We consider two different maps: one that uses an inclusion of complexes (see Definition \ref{inclusionassembly}) and a second version that is better adapted to the free assembly map defined on $K$-theory.

\subsection{The assembly map $\mu_{\rm dR,\phi}$ and injectivity}
First, we consider the assembly mapping $\mu_{\rm dR,\phi}:H_*^{\rm rel}(B\phi)\to \hat{H}_*^{\rm rel}(\phi_{\mathcal{A}})$ defined from the inclusion $\Omega_*^{\rm <e>}(\phi)\hookrightarrow \hat{\Omega}_*^{\rm ab}(\phi_\mathcal{A})$ and the isomorphism $\beta_{\rm rel}$. It is unclear to the authors if this version of the assembly mapping is compatible with the free assembly mapping and the Chern characters on a relative level. However, under minor assumptions, they are indeed compatible in the absolute case. Recall the definition of an admissible completion from Definition \ref{admissibleextdef}, see page \pageref{admissibleextdef}. 

For a Fr\' echet algebra $\mathcal{A}$, we let $\hat{H}C^*(\mathcal{A})$ denote its topological cyclic cohomology, i.e. the topological cohomology of the complex of continuous cyclic cocycles. Note that if the algebra $\mathcal{A}$ is admissible, the composition 
$$\hat{H}C^*(\mathcal{A})\to HC^*(\C[\Gamma])\xrightarrow{(\beta^{-1}_\Gamma)^*} H^*(B\Gamma)\otimes HC^*(\C)\to H^*(B\Gamma)\otimes HC^0(\C)=H^*(B\Gamma),$$ 
is surjective. There is a pairing $\hat{H}C^*(\mathcal{A})\times \hat{H}^{\rm dR}_*(\mathcal{A})\to \C$  obtained from noting that Proposition \ref{hdrandhcyc} (see page \pageref{hdrandhcyc}) is continuous and defines an inclusion of $\hat{H}^{\rm dR}_*(\mathcal{A})$ into continuous cyclic homology $\hat{H}C_*(\mathcal{A})$. We say that $\hat{H}C^*(\mathcal{A})$ separates $\hat{H}^{\rm dR}_*(\mathcal{A})$ if the pairing is non-degenerate.

\begin{prop}
\label{propcompass}
Let $\Gamma$ be a discrete finitely generated group, $B\Gamma$ a locally finite $CW$-complex and $\mathcal{A}\supseteq \C[\Gamma]$ a Fr\'echet algebra completion. 
\begin{enumerate}
\item If $\hat{H}C^*(\mathcal{A})$ separates $\hat{H}^{\rm dR}_*(\mathcal{A})$, the following diagram commutes
\begin{center}
$\begin{CD}
K^{\textnormal{geo}}_*(B \Gamma) @>\mu_{\textnormal{geo}} >> K^{\textnormal{geo}}_*(pt;\mathcal{A})  \\
@V\mathrm{ch}_*VV @VV\mathrm{ch}^{\mathcal{A}}_*V   \\
H_*(B\Gamma)@>\mu_{\rm dR} >> \hat{H}^{\rm dR}_*(\mathcal{A}) \\
\end{CD}$
\end{center}
\item If moreover $\mathcal{A}$ is admissible, $\mu_{\rm geo}$ and $\mu_{\rm dR}$ are rationally injective. In particular, $\Gamma$ satisfies the $\epsilon$-strong Novikov conjecture if it admits a Fr\'echet algebra completion admissible for a $C^*$-completion $C^*_\epsilon(\Gamma)$.
\end{enumerate}
\end{prop}

\begin{proof}
By our assumption on $\Gamma$, $\ch_{B\Gamma}:K^{\textnormal{geo}}_*(B \Gamma) \to H_*(B\Gamma)$ is a complex isomorphism. Therefore, to prove (1), it suffices to prove that $\langle c,\ch^\mathcal{A}_*(\mu_{\rm geo}(x))\rangle= \langle c,\mu_{\rm dR}(\ch_*(x))\rangle$ for all $x=(M,E,f)\in K^{\textnormal{geo}}_*(B \Gamma)$ and $c\in \hat{H}C^*(\mathcal{A})$. This fact follows from Connes-Moscovici's higher index theorem \cite[Proposition 6.3]{cmnovikov} which together with Theorem \ref{theabschen} gives us the required identity in the following argument $$\langle c,\ch^\mathcal{A}_*(\mu_{\rm geo}(x))\rangle=\int_M \ch(E)\wedge \mathrm{Td}(M)\wedge (\beta^{-1}_\Gamma)^*(c)=\langle c,\mu_{\rm dR}(\ch_*(x))\rangle.$$

To prove (2), we note that for any cocycle $c\in C^*_\mathcal{A}(\Gamma)$, $\langle \hat{c},\mu_{\rm dR}(\ch_*(x))\rangle=\langle c,\ch_*(x)\rangle$ where $\hat{c}$ is defined as in equation \eqref{thereisahat}. Since $\mathcal{A}$ is admissible, it follows that $\ch_*(x)=0$ if $\mu_{\rm dR}(\ch_*(x))=0$. Since the image of the Chern character spans $H_*(B\Gamma)$, $\mu_{\rm dR}$ is injective, and since $\ch_*$ is rationally injective for locally finite $CW$-complexes, so is $\mu_{\rm geo}$.
\end{proof}

\subsection{Another approach to assembly in homology}
The second approach to assembly uses the connection on the Mishchenko bundle. As above, we consider a group homomorphism $\phi:\Gamma_1\to \Gamma_2$. For a manifold with boundary $W$ equipped with mappings $f:W\to B\Gamma_2$ and $g:\partial W\to B\Gamma_1$ such that $f|_{\partial W}=B\phi\circ g$, we define a mapping 
\begin{align}
\nonumber
I_{f,g}:\Omega^*(W)&:=C^\infty (W,\wedge^*T^*W)\to \Omega_*^{{\rm rel}<e>}(\phi), \\
\label{ifgdef}
& I_{f,g}(\nu):= \begin{pmatrix}\int_W\nu \wedge\omega_{(W,(f,g))}\\
{}\\\int_{\partial W}\nu \wedge\omega_{(\partial W,g)}\end{pmatrix},
\end{align}
where $\omega_{(W,(f,g))}$ and $\omega_{(\partial W,g)}$ are as in Example \ref{mishcurv}. This is a well defined mapping by Example \ref{mishcurv}. The following proposition follows from Stokes' theorem and the fact that the total curvatures are $\mathrm{d}_W+\mathrm{d}_{\mathcal{A}}$-closed.

\begin{prop}
We have the identity that 
$$\int_W\mathrm{d} \nu \wedge\omega_{(W,(f,g))}=\phi_*\left(\int_{\partial W}\nu \wedge\omega_{(\partial W,g)}\right)+\mathrm{d}_\mathcal{A} \int_W \nu \wedge\omega_{(W,(f,g))}.$$
In particular, the mapping $I_{f,g}$ is a morphism of complexes.
\end{prop}

\begin{remark}
\label{remarkonifg}
If $(W,E,(f,g))$ is a cycle for $K_*^{\rm geo}(B\phi)$, the argument in Example \ref{mishcurv} shows that $\ch_{\rm rel}^{\phi_\mathcal{A}}(\mu_{\rm geo}(W,E,(f,g)))$ is represented by $I_{f,g}(\ch(\nabla_E)\wedge \mathrm{Td}(W))$ for a hermitean connection $\nabla_E$ on $E$ and a Riemannian structure on $W$, both being of product type near $\partial W$.
\end{remark}

\begin{define}
\label{definitionofmishass}
Define $\mu_{\rm M}^{\phi_\mathcal{A}}:H_*(B\phi)\to \hat{H}_*^{\rm rel}(\phi_\mathcal{A})$ as follows. For $x\in H_*(B\phi)$ of the form $(f,g)_*(\nu\cap [W])$, where $[W]\in H_*(W,\partial W)$ is the fundamental class of a manifold with boundary, $\nu$ a closed differential form on $W$ and $(f,g):[i:\partial W\to W]\to [B\phi:B\Gamma_1\to B\Gamma_2]$, we define 
$$\mu_{\rm M}^{\phi_\mathcal{A}}(x):=[I_{f,g}(\nu)].$$
We write $\mu_{\rm M}^{\mathcal{A}}$ for the mapping in the absolute case. 
\end{define}

\begin{remark}
By the same proof as that of Proposition \ref{propcompass}.1, $\mu_{\rm M}=\mu_{\rm dR}$ in the absolute case if $\hat{H}C^*(\mathcal{A})$ separates $\hat{H}^{\rm dR}_*(\mathcal{A})$.
\end{remark}

\begin{theorem}
\label{thmonassembly}
The map $\mu_{\rm M}^{\phi_\mathcal{A}}$ is well defined and fits into a commuting diagram with exact rows:
\begin{center}
\scriptsize
$\begin{CD}
@>>> H_*(B\Gamma_1) @>B\phi_*>> H_*(B\Gamma_2) @>>> H_*^{\rm rel}(B\phi) @>>> H_{*-1}(B\Gamma_1) @>>> \\
@. @VV\mu_{\rm M}^{\mathcal{A}_1}V @VV\mu_{\rm M}^{\mathcal{A}_2}V  @VV\mu_{\rm M}^{\phi_\mathcal{A}} V  @VV\mu_{\rm M}^{\mathcal{A}_1} V\\
@>>> \hat{H}_*^{\rm dR}(\mathcal{A}_1) @>(\phi_\mathcal{A})_*>>\hat{H}_*^{\rm dR}(\mathcal{A}_1) @>>> \hat{H}_*^{\rm rel}(\phi_\mathcal{A}) @>>> \hat{H}_*^{\rm dR}(\mathcal{A}_1)@>>> \\
\end{CD}$
\end{center}
Moreover, the mapping $\mu_{\rm M}$ makes the following diagram commutative:
\begin{equation}
\label{diagraminthmonassembly}
\begin{CD}
K^{\textnormal{geo}}_*(B \phi) @>\mu_{\textnormal{geo}} >> K^{\textnormal{geo}}_*(pt;\phi_{\mathcal{A}})  \\
@V\mathrm{ch}^{B\phi}_*V
V @VV\mathrm{ch}^{\phi_{\mathcal{A}}}_*V   \\
H_*(B\phi)@>\mu_{\rm M}^{\phi_\mathcal{A}} >> \hat{H}^{\rm rel}_*(\phi_{\mathcal{A}}) \\
\end{CD}
\end{equation}
\end{theorem}

\begin{proof}
The Chern characters $\ch:K_*^{\rm geo}(B\Gamma_i)\to H_*(B\Gamma_i)$ and $\ch_{\rm rel}^{B\phi}:K_*^{\rm geo}(B\phi)\to H_*^{\rm rel}(B\phi)$ are complex isomorphisms by Theorem \ref{cheronhom}. Hence, it follows that $\mu_{\rm M}^{\phi_\mathcal{A}}$ is well defined. Commutativity of the first diagram follows from the computations of Remark \ref{comattcuelc} (see page \pageref{comattcuelc}) and the definition of $I_{f,g}$ in Equation \eqref{ifgdef}. The commutativity of the diagram \eqref{diagraminthmonassembly} is immediate from the construction and Remark \ref{remarkonifg}.
\end{proof}

\begin{lemma}
\label{autolemma}
Let $\phi:\Gamma_1\to \Gamma_2$ be a group homomorphism. There is an automorphism $\zeta_\phi$ of $H_*(B\phi)$ such that $\mu_{\rm M}^{\phi_\mathcal{A}}=\mu_{\rm dR}^{\phi_\mathcal{A}}\circ \zeta_\phi$ for any Frechet algebra completions to which $\phi$ extends.
\end{lemma}

\begin{proof}
By Remark \ref{remarkonrelloc}, there is an isomorphism $\beta_{\rm rel}:H_*^{{\rm rel}, <e>}(\phi)\to H_*(B\phi)\otimes HC_*(\C)$. It follows from Example \ref{mishcurv} that Definition \ref{definitionofmishass} of $\mu_{\rm M}^{\phi_\mathcal{A}}$ indeed extends to an $HC_*(\C)$-linear mapping $\xi_\phi$ on $H_*^{{\rm rel}, <e>}(\phi)$ via $\beta_{\rm rel}$. Here we are implicitly using that $H_*^{{\rm rel}, <e>}(\phi)\subseteq H_*^{{\rm rel}}(\phi)$ is a direct summand. By construction, $\mu_{\rm M}^{\phi_\mathcal{\C[\Gamma]}}\circ \beta_{\rm rel}=\xi_\phi$. It follows from Connes-Moscovici's higher index theorem \cite[Proposition 6.3]{cmnovikov} that $\xi_\phi$ is the identity in the absolute case, so a five lemma argument implies that $\xi_\phi$ is an isomorphism. Since $\xi_\phi$ is $HC_*(\C)$-linear, $\xi_\phi$ induces an automorphism $\zeta_\phi$ of $H_*(B\phi)$ via $\beta_{\rm rel}$.
\end{proof}

\section{The strong relative Novikov property}

The ideas of \cite{cmnovikov}, recalled above in Proposition \ref{propcompass}.(2), can be applied in the relative setting to prove Novikov type properties for certain group homomorphisms. The key assumption is a relative admissibility condition similar to Definition \ref{admissibleextdef} on page \pageref{admissibleextdef}. We prove that a homomorphism of hyperbolic groups that extends to the reduced group $C^*$-algebras has the reduced strong Novikov property in Theorem \ref{rnovikovforhyp}.

\subsection{General results}
\label{genressec6}

\begin{define}
Let $\Gamma_i$ be groups and $\mathcal{A}_i$ Fr\'echet *-algebra completions of $\C[\Gamma_i]$, $i=1,2$. The pair $(\mathcal{A}_1,\mathcal{A}_2)$ is said to be admissible for a group homomorphism $\phi:\Gamma_1\to \Gamma_2$ if 
\begin{enumerate}
\item[i)] $\phi$ extends to a continuous $*$-homomorphism $\phi_\mathcal{A}:\mathcal{A}_1\to \mathcal{A}_2$, and
\item[ii)] any class $z\in H^*(B\phi)$ can be represented by a cocycle $(c_1,c_2)\in \mathcal{C}^*_{\rm rel}(B\phi)$ such that the mapping $(\hat{c}_1,\hat{c}_2):\Omega^{\rm rel}(\phi)\to \C$ extends by continuity to a mapping $\hat{H}^{\rm rel}_*(\phi_\mathcal{A})\to \C$.
\end{enumerate}
\end{define}

\begin{theorem}
\label{theoremonabscond}
If $(\mathcal{A}_1,\mathcal{A}_2)$ is admissible for a homomorphism $\phi$ between groups $\Gamma_1$ and $\Gamma_2$, and $B\Gamma_1$ and $B\Gamma_2$ locally finite $CW$-complexes, the assembly mapping $\mu_{\rm geo}^{\phi_\mathcal{A}}:K_*^{\rm geo}(B\phi)\to K_*^{\rm geo}(pt; \phi_\mathcal{A})$ is a rational injection.
\end{theorem}

\begin{proof}
Since $\ch_*^{B\phi}$ is a rational injection, the diagram \eqref{diagraminthmonassembly} in Theorem \ref{thmonassembly} and Lemma \ref{autolemma} implies that it suffices to prove that if $\mu_{\rm dR}^{\phi_\mathcal{A}}(\ch_*^{B\phi}(x))=0$ for an $x\in K_*^{\rm geo}(B\phi)$ then $\ch_*^{B\phi}(x)=0$. However, if $\mu_{\rm dR}^{\phi_\mathcal{A}}(\ch_*^{B\phi}(x))=0$ then for any cocycle $(c_1,c_2)\in \mathcal{C}^*_{\rm rel}(B\phi)$ such that the mapping $(\hat{c}_1,\hat{c}_2):\Omega^{\rm rel}(\phi)\to \C$ extends by continuity to a mapping $\hat{H}^{\rm rel}_*(\phi_\mathcal{A})\to \C$ it holds that 
$$\left\langle \left[(\hat{c}_1,\hat{c}_2)^T\right],\mu_{\rm dR}^{\phi_\mathcal{A}}(\ch_*^{B\phi}(x))\right\rangle=\left\langle \left[(c_1,c_2)^T\right],\ch_*^{B\phi}(x)\right\rangle=0.$$
We conclude that $\left\langle z,\ch_*^{B\phi}(x)\right\rangle=0$ for any class $z\in H^*(B\phi)$ because $(\mathcal{A}_1,\mathcal{A}_2)$ is admissible for the homomorphism $\phi$. The theorem follows.
\end{proof}

\begin{cor}
\label{cornoviko}
Suppose that $\phi$ is a homomorphism between groups $\Gamma_1$ and $\Gamma_2$ admitting  an admissible pair $(\mathcal{A}_1,\mathcal{A}_2)$ such that the inclusion $\mathcal{A}_i\subseteq C^*_{\epsilon}(\Gamma_i)$ induces a rational isomorphism on $K$-theory for some completion $\epsilon$. If $\phi$ extends to a $*$-homomorphism $C^*_\epsilon(\Gamma_1)\to C^*_\epsilon(\Gamma_2)$, then $\phi:\Gamma_1\to \Gamma_2$ has the $\epsilon$-strong relative Novikov property.
\end{cor}

\subsection{Hyperbolic groups}

\begin{theorem}
\label{rnovikovforhyp}
A homomorphism of hyperbolic groups continuous in the reduced $C^*$-norm has the reduced strong relative Novikov property.
\end{theorem}

Denote the hyperbolic groups by $\Gamma_1$ and $\Gamma_2$ and the homomorphism by $\phi$. We will use the dense subalgebras $\ell^1(\Gamma_i)\subseteq C^*_{\bf red}(\Gamma_i)$, $i=1,2$, and the results of Subsection \ref{genressec6}. Hyperbolic groups satisfy the Bost conjecture and the Baum-Connes conjecture, see \cite{laffa,mineyevandyu}, so the inclusions $\ell^1(\Gamma_i)\subseteq C^*_{\bf red}(\Gamma_i)$, $i=1,2$, induce isomorphisms on $K$-theory. Therefore, the theorem follows from Corollary \ref{cornoviko} if we can prove that $(\ell^1(\Gamma_1),\ell^1(\Gamma_2))$ is an admissible pair for the homomorphism $\phi$. 

We denote the mapping cone of $\phi^*:\mathcal{C}^*_{\ell^1}(\Gamma_2)\to \mathcal{C}^*_{\ell^1}(\Gamma_2)$ by $\mathcal{C}^*_{\ell^1}(\phi)$. There is an inclusion $\mathcal{C}^*_{\ell^1}(\phi)\subseteq \mathcal{C}^*(\phi)$, we call the elements in this image bounded. We can conclude that $(\ell^1(\Gamma_1),\ell^1(\Gamma_2))$ is an admissible pair for $\phi$ once any cocycle $c\in \mathcal{C}^*(\phi)$ is cohomologous to a bounded cocycle by Example \ref{hyperelleone}. The cohomology will be denoted by $H^*_{\ell^1}(\phi):=H^*(\mathcal{C}^*_{\ell^1}(\phi))$ and we consider the mapping $H^*_{\ell^1}(\phi)\to H^*(B\phi)$ induced by the inclusion of complexes. 

\begin{lemma}
\label{admisforle}
Let $\phi:\Gamma_1\to \Gamma_2$ be a homomorphism of hyperbolic groups. The canonical mapping $H^*_{\ell^1}(\phi)\to H^*_{\rm rel}(B\phi)$ is an isomorphism. In particular, any cocycle $c\in \mathcal{C}^*(\phi)$ is cohomologous to a bounded cocycle and $(\ell^1(\Gamma_1),\ell^1(\Gamma_2))$ is an admissible pair for $\phi$.  
\end{lemma}

\begin{proof}
By construction, we have a short exact sequence of complexes
$$0\to \mathcal{C}^{*+1}_{\ell^1}(\Gamma_1)\to \mathcal{C}^*_{\ell^1}(\phi)\to \mathcal{C}^{*}_{\ell^1}(\Gamma_2)\to 0.$$
By naturality, we arrive at a commuting diagram with exact rows:
\begin{center}
\scriptsize
$\begin{CD}
@<<< H^*_{\ell^1}(\Gamma_1) @<\phi^*<< H^*_{\ell^1}(\Gamma_2) @<<< H^*_{\ell^1}(\phi) @<<< H^{*+1}_{\ell^1}(\Gamma_1) @<<< \\
@. @VVV @VVV  @VVV  @VVV\\
@<<< H^*(B\Gamma_1) @<B\phi^*<< H^*(B\Gamma_2) @<<< H^*_{\rm rel}(B\phi) @<<< H^{*+1}(B\Gamma_1) @<<< \\
\end{CD}$
\end{center}
The lemma now follows from Proposition \ref{hyperiso} and the five lemma.
\end{proof}

\begin{remark}
We remark that by Remark \ref{strongremark}, the relative Novikov conjecture (homotopy invariance of signatures of manifolds with boundaries) holds for hyperbolic groups. 
\end{remark}

\subsection{Groups with polynomially bounded cohomology in $\C$}

In this subsection we consider a polynomially bounded homomorphism $\phi:\Gamma_1\to \Gamma_2$ between groups with polynomially bounded cohomology in $\C$ (see Definition \ref{pcandpolbdd}).

\begin{theorem}
\label{polbddthm}
Let $\Gamma_1$ and $\Gamma_2$ be two groups with polynomially bounded cohomology in $\C$ and $\phi:\Gamma_1\to \Gamma_2$ a polynomially bounded homomorphism. Then the free geometric assembly mapping $\mu_{\rm geo}^{\phi_{\ell^1}}:K_*^{\rm geo}(B\phi)\to K_*^{\rm geo}(pt; \phi_{\ell^1})$ is a rational injection.
\end{theorem}

The proof of Theorem \ref{polbddthm} goes along similar lines as that of Theorem \ref{rnovikovforhyp}, so we only provide the essential points. Since $\phi$ is polynomially bounded, it induces a continuous $*$-homomorphism $\phi_{H^{1,\infty}_L}:H^{1,\infty}_{L_1}(\Gamma_1)\to H^{1,\infty}_{L_2}(\Gamma_2)$. Recall the fact stated in Example \ref{ageoneess} that $H^{1,\infty}_{L_i}(\Gamma_i)\hookrightarrow \ell^1(\Gamma_i)$ is closed under holomorphic functional calculus (see \cite[Proposition 2.3]{jolissatwo}). A five lemma argument shows that the inclusions $H^{1,\infty}_{L_i}(\Gamma_i)\hookrightarrow \ell^1(\Gamma_i)$ compatible with $\phi$ induce an isomorphism $K_*^{\rm geo}(pt; \phi_{H^{1,\infty}_L})\cong K_*^{\rm geo}(pt; \phi_{\ell^1})$, and we reduce the proof to showing that 
$$\mu_{\rm geo}^{\phi_{H^{1,\infty}_L}}:K_*^{\rm geo}(B\phi)\to K_*^{\rm geo}(pt; \phi_{H^{1,\infty}_L}),$$
is rationally injective. Using the fact that $\Gamma_1$ and $\Gamma_2$ have polynomially bounded cohomology in $\C$, $H^*_{H^{1,\infty}_L}(\Gamma_i)\cong H^*(B\Gamma_i)$, and an argument involving the five lemma, as in Lemma \ref{admisforle}, shows that $(H^{1,\infty}_L(\Gamma_1),H^{1,\infty}_L(\Gamma_2))$ is admissible for $\phi$. Now, Theorem \ref{polbddthm} follows from Theorem \ref{theoremonabscond}.

\vspace{0.25cm}
\noindent
Email address: robin.deeley@gmail.com \vspace{0.25cm} \\
{ \footnotesize Department of Mathematics,
University of Colorado Boulder,
Campus Box 395,
Boulder, CO 80309-0395,
USA} \vspace{.5cm}\\
Email address: goffeng@chalmers.se \vspace{0.25cm} \\
{ \footnotesize Department of Mathematical Sciences, University of Gothenburg and Chalmers University of Technology, Chalmers Tv\"argata 3, 412 96 G\"oteborg, Sweden }
\end{document}